\documentclass[12pt]{article}

\title{\Large\bf 
A presentation of the torus-equivariant \\[2mm]
quantum $K$-theory ring of flag manifolds of type $A$, \\[2mm]
Part I\hspace{-1.5pt}I: quantum double Grothendieck polynomials%
\footnote{Key words and phrases: (quantum) $K$-theory, 
(quantum) Schubert calculus, (quantum) double Grothendieck polynomials, 
semi-infinite flag manifold. \newline
2020 Mathematics Subject Classification: Primary 14N15, 14N35;
Secondary 14M15, 05E14, 05E05.}%
}
\author{%
Toshiaki Maeno \\
 \small Department of Mathematics, Faculty of Science and Technology, Meijo University, \\
 \small 1-501 Shiogamaguchi, Tempaku-ku, Nagoya 468-8502, Japan \\
 \small (e-mail: {\tt tmaeno@meijo-u.ac.jp}) \\[5mm]
Satoshi Naito \\ 
 \small Department of Mathematics, Tokyo Institute of Technology, \\
 \small 2-12-1 Oh-okayama, Meguro-ku, Tokyo 152-8551, Japan \\
 \small (e-mail: {\tt naito@math.titech.ac.jp}) \\[5mm]
and \\[3mm]
Daisuke Sagaki \\ 
 \small Department of Mathematics, \\
 \small Faculty of Pure and Applied Sciences, University of Tsukuba, \\
 \small 1-1-1 Tennodai, Tsukuba, Ibaraki 305-8571, Japan \\
 \small (e-mail: {\tt sagaki@math.tsukuba.ac.jp})
}
\date{}

\usepackage{amsmath, amssymb, amsthm, amscd}
\usepackage{bm}
\usepackage{mathrsfs}
\usepackage{xcolor}
\usepackage{ulem}

\makeatletter
\renewcommand\section{\@startsection{section}{1}{0pt}
{-3.5ex plus -1ex minus -.2ex}{1.0ex plus .2ex}{\large\bf}}
\renewcommand\subsection{\@startsection{subsection}{1}{0pt}
{2.5ex plus 1ex minus .2ex}{-1em}{\bf}}
\makeatother

\numberwithin{equation}{section}

\theoremstyle{plain}
\newtheorem{thm}{Theorem}[section]
\newtheorem{lem}[thm]{Lemma}
\newtheorem{prop}[thm]{Proposition}
\newtheorem{cor}[thm]{Corollary}

\newtheorem{ithm}{Theorem}

\theoremstyle{definition}
\newtheorem{dfn}[thm]{Definition}

\theoremstyle{remark}

\newtheorem{rem}[thm]{Remark}

\newcommand{\BZ}{\mathbb{Z}}
\newcommand{\BC}{\mathbb{C}}
\newcommand{\BR}{\mathbb{R}}
\newcommand{\BP}{\mathbb{P}}

\newcommand{\CO}{\mathcal{O}}
\newcommand{\CF}{\mathcal{F}}
\newcommand{\CI}{\mathcal{I}}
\newcommand{\CA}{\mathcal{A}}
\newcommand{\CU}{\mathcal{U}}

\newcommand{\be}{\mathbf{e}}
\newcommand{\bp}{\mathbf{p}}
\newcommand{\bq}{\mathbf{q}}
\newcommand{\bs}{\mathbf{s}}

\newcommand{\bI}{\mathbf{I}}
\newcommand{\bD}{\mathbf{D}}

\newcommand{\bchi}{\bm{\chi}}

\newcommand{\st}{\mathsf{t}}

\newcommand{\SF}{\mathsf{F}}

\newcommand{\RA}{\mathrm{A}}
\newcommand{\RB}{\mathrm{B}}
\newcommand{\RC}{\mathrm{C}}

\newcommand{\RX}{\mathrm{X}}

\newcommand{\Qq}{\mathscr{Q}}

\newcommand{\Fg}{\mathfrak{g}}
\newcommand{\Fh}{\mathfrak{h}}
\newcommand{\FG}{\mathfrak{G}}
\newcommand{\FF}{\mathfrak{F}}

\newcommand{\Fsl}{\mathfrak{sl}}

\newcommand{\vpi}{\varpi}
\newcommand{\eps}{\epsilon}

\newcommand{\lng}{w_{\circ}}

\newcommand{\af}{\mathrm{af}}

\DeclareMathOperator{\ed}{end}
\DeclareMathOperator{\wt}{wt}
\DeclareMathOperator{\id}{id}

\DeclareMathOperator{\dnn}{down}
\DeclareMathOperator{\sgn}{sgn}

\newcommand{\QG}{\mathbf{Q}_{G}}
\newcommand{\QGr}{\mathbf{Q}_{G}^{\mathrm{rat}}}

\newcommand{\LQG}[1]{[\CO_{\QG}(#1)]}
\newcommand{\SQG}[1]{[\CO_{\QG(#1)}]}

\newcommand{\Hom}{\mathrm{Hom}}
\newcommand{\Par}{\mathrm{Par}}
\newcommand{\Frac}{\mathrm{Frac}}
\newcommand{\QBG}{\mathrm{QBG}}

\newcommand{\bra}[1]{[\![#1]\!]}
\newcommand{\pra}[1]{(\!(#1)\!)}
\newcommand{\pair}[2]{\langle #1, #2 \rangle}

\newcommand{\edge}[1]{ \xrightarrow{\hspace{2pt}#1\hspace{2pt}} }
\newcommand{\Qe}[1]{ \xrightarrow[\mathsf{Q}]{\hspace{2pt}#1\hspace{2pt}} }
\newcommand{\Be}[1]{ \xrightarrow[\mathsf{B}]{\hspace{2pt}#1\hspace{2pt}} }

\newcommand{\ol}[1]{\overline{#1}}

\newcommand{\sprod}{\sideset{}{^\star}\prod}

\newenvironment{enu}{%
 \begin{enumerate}%
}{\end{enumerate}}

\allowdisplaybreaks[4]

\begin{document}

\maketitle

%================%
% START ABSTRACT %
%================%

\begin{abstract}
In our previous paper, we gave a presentation of the torus-equivariant quantum $K$-theory ring $QK_{H}(Fl_{n+1})$ 
of the (full) flag manifold $Fl_{n+1}$ of type $A_{n}$ as a quotient of a polynomial ring by an explicit ideal. 
In this paper, we prove that quantum double Grothendieck polynomials, introduced by Lenart-Maeno, 
represent the corresponding (opposite) Schubert classes 
in the quantum $K$-theory ring $QK_{H}(Fl_{n+1})$ under this presentation. 
The main ingredient in our proof is an explicit formula expressing the semi-infinite Schubert class 
associated to the longest element of the finite Weyl group, which is proved by making use of 
the general Chevalley formula for the torus-equivariant $K$-group of 
the semi-infinite flag manifold associated to $SL_{n+1}(\BC)$.
\end{abstract}

%==================%
% START SECTION 01 %
%==================%

\section{Introduction.} 
\label{sec:intro}
Let $Fl_{n+1}$ denote the (full) flag manifold $G/B$ of type $A_{n}$, 
where $G = SL_{n+1}(\BC)$ is the connected, 
simply connected simple algebraic group of type $A_{n}$ 
over the complex numbers $\BC$, with Borel subgroup $B$ consisting of 
the upper triangular matrices in $G = SL_{n+1}(\BC)$ 
and maximal torus $H \subset B$ consisting of the diagonal matrices in $G = SL_{n+1}(\BC)$. 
We denote by 
$QK_{H}(Fl_{n+1}) := K_{H}(Fl_{n+1}) \otimes_{R(H)} R(H)\bra{Q}$ 
the $H$-equivariant quantum $K$-theory ring, 
defined by Givental \cite{Giv} and Lee \cite{Lee}, 
where $K_{H}(Fl_{n+1}) = \bigoplus_{w \in W} R(H)[\CO^{w}]$ 
denotes the $H$-equivariant (ordinary) $K$-theory ring of $Fl_{n+1}$ 
with the (opposite) Schubert classes $[\CO^{w}]$ indexed by 
the elements $w$ of the finite Weyl group $W = S_{n+1}$ of $G = SL_{n+1}(\BC)$ 
as a basis over $R(H)$, and where $R(H)\bra{Q} = R(H)\bra{Q_1, \ldots, Q_{n}}$ 
denotes the ring of formal power series 
in the Novikov variables $Q_{i} := Q^{\alpha_i^{\vee}}$ corresponding 
to the simple coroots $\alpha_i^{\vee}$, $1 \leq i \leq n$, 
with coefficients in the representation ring $R(H)$ of $H$;
we will identify the representation ring $R(H)$ with 
the group algebra $\BZ[P] = \bigoplus_{\nu \in P} \BZ \be^{\nu}$ of 
the weight lattice $P = \sum_{i =1}^{n} \BZ \vpi_{i}$ of $G = SL_{n+1}(\BC)$, 
where $\vpi_{i}$, $1 \leq i \leq n$, are the fundamental weights; 
we set $\varpi_{0} := 0$ and $\varpi_{n+1} := 0$ by convention. 

In our previous paper \cite{MaNS}, we proved that there exists 
an $R(H)\bra{Q}$-algebra isomorphism 
\begin{equation} \label{eq:iPsiQ}
\Psi^{Q}:R(H)\bra{Q}[x_{1},\dots,x_{n},x_{n+1}]/\CI^{Q}
\stackrel{\sim}{\rightarrow} QK_{H}(Fl_{n+1}),
\end{equation}
where $\CI^{Q}$ is the ideal of 
$R(H)\bra{Q}[x_{1},\dots,x_{n},x_{n+1}]$ 
generated by 
\begin{equation*}
\sum_{
   \begin{subarray}{c}
   J \subset [n+1] \\[1mm]
   |J|=l
   \end{subarray}} \, 
\prod_{j \in J, j+1 \notin J} (\overline{1-Q_{j}})
\prod_{j \in J} (1-x_{j}) - 
\sum_{
   \begin{subarray}{c}
   J \subset [n+1] \\[1mm]
   |J|=l
   \end{subarray}}\be^{\eps_{J}}, 
\qquad 1 \le l \le n+1, 
\end{equation*}
with $[n+1] := \{1,2,\dots,n+1\}$, and 
\begin{equation*}
\ol{1-Q_{j}}:=
\begin{cases}
1-Q_{j} & \text{if $1 \le j \le n$}, \\
1 & \text{if $j=n+1$}; 
\end{cases}
\end{equation*}
here, for a subset $J \subset [n+1]$, we set $\eps_{J} := \sum_{j \in J} \eps_{j}$, 
where $\eps_{j} = \vpi_{j} - \vpi_{j-1}$ for $1 \leq j \leq n+1$. 
Also, $\Psi^{Q}$ maps the residue class of 
$(\ol{1-Q_j})(1-x_{j})$ modulo $\CI^{Q}$ to $[\CO_{Fl_{n+1}}(-\eps_{j})]$ for $1 \le j \le n+1$.  
Here, for $1 \leq j \leq n+1$, $\CO_{Fl_{n+1}}(-\eps_{j})$ denotes 
the quotient (line) bundle $\CU_{j}/\CU_{j-1}$ over $Fl_{n+1}$, where 
$0 = \CU_{0} \subset \CU_{1} \subset \cdots \subset \CU_{n} \subset \CU_{n+1} = Fl_{n+1} \times \BC^{n+1}$ 
is the universal, or tautological, flag of subvector bundles of the trivial bundle $FL_{n+1} \times \BC^{n+1}$. 

The purpose of this paper is to prove that quantum double Grothendieck polynomials, 
introduced in \cite[Section~8]{LM}, represent the corresponding (opposite) 
Schubert classes in $QK_{H}(Fl_{n+1})$ under the presentation above; 
this result can be thought of as the $H$-equivariant analog of \cite[Theorem~50]{LNS}. 
To be more precise, for the longest element $\lng \in W=S_{n+1}$, 
the associated quantum double Grothendieck polynomial $\FG_{\lng}^{Q}(x,y)$ is defined as: 
\begin{equation*}
\FG_{\lng}^{Q}(x,y) =
\prod_{k=1}^{n}\Bigg( \sum_{l=0}^{k}(-1)^{l}(1-y_{n+1-k})^{l}F_{l}^{k}(x_{1},\dots,x_{n},x_{n+1}) \Bigg),
\end{equation*}
where 
\begin{equation*}
F_{l}^{k}(x_{1},\dots,x_{n},x_{n+1}):=
\sum_{
   \begin{subarray}{c}
   J \subset [k] \\[1mm]
   |J|=l
   \end{subarray}} \, 
\prod_{j \in J, j+1 \notin J} (1-Q_{j}) \prod_{j \in J}(1-x_{j})
\end{equation*}
for $0 \le l \le k \le n$, with 
$[k] := \{1, 2, \ldots, k\}$; 
note that $\FG_{\lng}^{Q}(x,y)$ is an element of 
$\BZ\bra{Q}[x_{1},\dots,x_{n},x_{n+1}][(1-y_{1})^{\pm 1},\dots,(1-y_{n})^{\pm 1}]$. 

The following is the key result in the proof of our main result; see Theorem~\ref{thm:lng}. 
%
%%%%%%%%%%%%%
%%% ithm1 %%%
%%%%%%%%%%%%%
%
\begin{ithm} \label{ithm1}
Let $\lng \in W=S_{n+1}$ be the longest element. 
Then, under the isomorphism \eqref{eq:iPsiQ}, the following equality holds in $QK_{H}(Fl_{n+1})$\,{\rm:}
\begin{equation}
\Psi^{Q}(\FG_{\lng}^{Q}(x,y) \ \mathrm{mod} \ \CI^{Q}) = [\CO^{\lng}] \in QK_{H}(Fl_{n+1}), 
\end{equation}
where we identify $R(H) \cong \BZ[\be^{\pm \eps_{1}},\dots,\be^{\pm \eps_{n}}]$
with $\BZ[(1-y_{1})^{\pm 1},\dots,(1-y_{n})^{\pm 1}]$, and then 
$R(H)\bra{Q}[x_{1},\dots,x_{n},x_{n+1}]$ with 
$\BZ\bra{Q}[x_{1},\dots,x_{n},x_{n+1}][(1-y_{1})^{\pm 1},\dots,(1-y_{n})^{\pm 1}]$ 
by $1 - y_{j}=\be^{-\eps_{j}}$ for $1 \le j \le n$.
\end{ithm}

Here we should mention that the result above for 
the longest element $\lng \in W = S_{n+1}$ is obtained from 
an explicit formula (Corollary~\ref{cor:wk-lng}) expressing 
the corresponding semi-infinite Schubert class $[\CO_{\QG(\lng)}]$ 
in the $H$-equivariant $K$-group $K_{H}(\QG)$ of 
the semi-infinite flag manifold $\QG$ associated to $G = SL_{n+1}(\BC)$; 
see Section~\ref{sec03} for details.

For an arbitrary $w \in W = S_{n+1}$, 
the associated quantum double Grothendieck polynomial $\FG_{w}^{Q}(x,y)$ is defined as: 
\begin{equation*}
\FG_{w}^{Q}(x,y) =\pi_{w\lng}^{(y)}\FG_{\lng}^{Q}(x,y) 
\in \BZ\bra{Q}[x_{1},\dots,x_{n},x_{n+1}][(1-y_{1})^{\pm 1},\dots,(1-y_{n})^{\pm 1}], 
\end{equation*}
where $\pi_{w\lng}^{(y)}$ denotes the Demazure operator acting on the $y$-variables 
(see Section~\ref{sec0401} for details). 
Then, by making use of quantum (dual) left Demazure operators $\delta_{i}^{\vee}$, $1 \leq i \leq n$, 
on $QK_{H}(Fl_{n+1})$, which act only on equivariant parameters, 
we can prove the following main result of this paper (see Theorem~\ref{thm:main}); 
this line of argument is indeed proposed in \cite[Section~8]{MNS}.
%
%%%%%%%%%%%%%
%%% ithm2 %%%
%%%%%%%%%%%%%
%
\begin{ithm} \label{ithm2}
Let $w$ be an arbitrary element of $W = S_{n+1}$. Then, under the isomorphism \eqref{eq:iPsiQ}, 
the following equality holds in $QK_{H}(Fl_{n+1})$\,{\rm:}
\begin{equation*}
\Psi^{Q}(\FG_{w}^{Q}(x,y) \ \mathrm{mod} \ \CI^{Q}) = [\CO^{w}], 
\end{equation*}
where we identify $R(H) \cong \BZ[\be^{\pm \eps_{1}},\dots,\be^{\pm \eps_{n}}]$
with $\BZ[(1-y_{1})^{\pm 1},\dots,(1-y_{n})^{\pm 1}]$, and then 
$R(H)\bra{Q}[x_{1},\dots,x_{n},x_{n+1}]$ with 
$\BZ\bra{Q}[x_{1},\dots,x_{n},x_{n+1}][(1-y_{1})^{\pm 1},\dots,(1-y_{n})^{\pm 1}]$ 
by $1 - y_{j}=\be^{-\eps_{j}}$ for $1 \le j \le n$. 
\end{ithm}

\subsection*{Acknowledgments.}
The second and third authors would like to 
thank Cristian Lenart and Daniel Orr for related collaborations.
% The second author would like to thank Leonardo C. Mihalcea for valuable discussions 
% on the presentation of the quantum $K$-theory ring of Grassmannians, 
% which inspired this work. 
S.N. was partly supported by JSPS Grant-in-Aid for Scientific Research (C) 21K03198. 
D.S. was partly supported by JSPS Grant-in-Aid for Scientific Research (C) 19K03145 and 23K03045.
%
%==================%
% START SECTION 02 %
%==================%
%
\section{An explicit formula expressing the semi-infinite Schubert class associated to $\lng$.} 
\label{sec:expr}

%=======================%
% START SUBSECTION 0201 %
%=======================%
%
\subsection{Algebraic groups.} \label{subsec:alggrp}
Let $G$ be a connected, simply-connected simple algebraic group over $\BC$, 
$H$ a maximal torus of $G$. 
Set $\Fg := \mathrm{Lie}(G)$ and $\Fh := \mathrm{Lie}(H)$. 
Thus $\Fg$ is a finite-dimensional simple Lie algebra over $\BC$ and 
$\Fh$ is a Cartan subalgebra of $\Fg$. 
We denote by $\pair{\cdot\,}{\cdot} : \Fh^{\ast} \times \Fh \rightarrow \BC$ 
the canonical pairing, where $\Fh^{\ast} = \Hom_{\BC}(\Fh, \BC)$. 

It is known that $\Fg$ has a root system $\Delta \subset \Fh^{\ast}$. 
We take the set $\Delta^{+} \subset \Delta$ of positive roots, 
and the set $\{ \alpha_{i} \}_{i \in I} \subset \Delta^{+}$ of simple roots. 
We denote by $\alpha^{\vee} \in \Fh$ the coroot corresponding to $\alpha \in \Delta$. 
Also, we denote by $\theta \in \Delta^+$ the highest root of $\Delta$, 
and set $\rho := (1/2) \sum_{\alpha \in \Delta^{+}} \alpha$. 
The root lattice $Q$ and the coroot lattice $Q^{\vee}$ of $\Fg$ are defined by 
$Q := \sum_{i \in I} \BZ \alpha_{i}$ and $Q^{\vee} := \sum_{i \in I} \BZ \alpha_{i}^{\vee}$. 

For $i \in I$, the weight $\vpi_{i} \in \Fh^{\ast}$ 
which satisfies $\pair{\vpi_{i}}{\alpha_{j}^{\vee}} = \delta_{i,j}$ for all $j \in I$, 
where $\delta_{i,j}$ denotes the Kronecker delta, 
is called the $i$-th fundamental weight. 
The weight lattice $P$ of $\Fg$ is defined by $P := \sum_{i \in I} \BZ \vpi_{i}$. 
We denote by $\BZ[P]$ the group algebra of $P$, that is, 
the associative algebra generated by formal elements $\be^{\nu}$, $\nu \in P$, 
where the product is defined by $\be^{\mu} \be^{\nu} := \be^{\mu + \nu}$ for $\mu, \nu \in P$. 

A reflection $s_{\alpha} \in GL(\Fh^{\ast})$, $\alpha \in \Delta$, 
is defined by $s_{\alpha} \mu := \mu - \pair{\mu}{\alpha^{\vee}} \alpha$ 
for $\mu \in \Fh^{\ast}$. We write $s_{i} := s_{\alpha_{i}}$ for $i \in I$. 
Then the (finite) Weyl group $W$ of $\Fg$ is defined to be the subgroup of $GL(\Fh^{\ast})$ 
generated by $\{ s_{i} \}_{i \in I}$, that is, $W := \langle s_{i} \mid i \in I \rangle$. 
For $w \in W$, there exist $i_{1}, \ldots, i_{r} \in I$ such that $w = s_{i_{1}} \cdots s_{i_{r}}$. 
If $r$ is minimal, then the product $s_{i_{1}} \cdots s_{i_{r}}$ is called a reduced expression for $w$, 
and $r$ is called the length of $w$; we denote by $\ell(w)$ the length of $w$. 
Note that a reduced expression for $w$ is not unique. However, the length is defined uniquely. 
Also, the affine Weyl group $W_{\af}$ of $\Fg$ is, by definition, the semi-direct product group 
$W \ltimes \{ t_{\xi} \mid \xi \in Q^{\vee} \}$ of $W$ and the abelian group 
$\{ t_{\xi} \mid \xi \in Q^{\vee} \} \cong Q^{\vee}$, 
where $t_{\xi}$ denotes the translation in $\Fh^{\ast}$ corresponding to $\xi \in Q^{\vee}$.

\begin{dfn} \label{dfn:QBG} % [{\cite[Definition~6.1]{BFP}}]
The quantum Bruhat graph of $W$, denoted by $\QBG(W)$, 
is the $\Delta^{+}$-labeled directed graph whose vertices are 
the elements of $W$ and whose edges are of the following form: 
$x \edge{\alpha} y$, with $x, y \in W$ and $\alpha \in \Delta^{+}$, 
such that $y = x s_{\alpha}$ and either of the following holds: 
(B) $\ell(y) = \ell (x) + 1$; 
(Q) $\ell(y) = \ell (x) + 1 - 2 \pair{\rho}{\alpha^{\vee}}$.
An edge satisfying (B) (resp., (Q)) is called a Bruhat edge (resp., quantum edge).
\end{dfn}

For an edge $x \edge{\alpha} y$ in $\QBG(W)$, 
we sometimes write $x \Be{\alpha} y$ (resp., $x \Qe{\alpha} y$) 
to indicate that the edge is a Bruhat (resp., quantum) edge. 

%
%=======================%
% START SUBSECTION 0202 %
%=======================%
%
\subsection{The root system of type $A$.}

We recall the root system of type $A$. In the rest of this paper, 
if we assume that $G$ is of type $A$, then we use the notation introduced in this subsection. 
Also, we set $[m]:=\{ 1,2,\dots,m \}$ for $m \in \BZ_{\ge 0}$.

Assume that $G$ is of type $A_{n}$, i.e., $G = SL_{n+1}(\BC)$. 
Then $\Fg = \Fsl_{n+1}(\BC)$, 
and $\Fh := \{ h \in \Fg \mid \text{$h$ is a diagonal matrix}\}$ is 
a Cartan subalgebra of $\Fg$. 
We let $\{\eps_{k} \mid k \in [n+1] \}$ be the standard basis of $\BZ^{n+1}$ 
and realize the weight lattice as $P=\BZ^{n+1}/\BZ(\eps_1+\cdots+\eps_{n+1})$. 
By abuse of notation, we continue to denote the image of $\eps_k$ in $P$ by the same symbol. 
Thus $\varpi_k := \eps_1+\dotsm+\eps_k$, for $k \in I = [n]$, are the fundamental weights of $\Fg$, 
and $\eps_1 + \cdots + \eps_n + \eps_{n+1} = 0$. 
We set $\alpha_{i} := \eps_{i} - \eps_{i+1}$ for $i \in I = [n]$ 
and $\alpha_{i, j} := \alpha_{i} + \alpha_{i+1} + \cdots + \alpha_{j}=\eps_{i}-\eps_{j+1}$ 
for $i, j \in [n]$ with $i \le j$. 
Then $\Delta = \{ \pm \alpha_{i,j} \mid 1 \le i \le j \le n \}$ forms a root system of $\Fg$, 
with the set of positive roots $\Delta^+ = \{ \alpha_{i,j} \mid 1 \le i \le j \le n \}$ 
and the set of simple roots $\{ \alpha_{1}, \ldots, \alpha_{n} \}$. 

Let us review the Weyl group of $\Fg=\Fsl_{n+1}(\BC)$. 
By the definition of the Weyl group, we have $W = \langle s_{1}, \ldots, s_{n} \rangle$, 
where $s_{1}, \ldots, s_{n}$ are simple reflections. 
For $i,j \in [n+1]$ with $i < j$, 
we denote by $(i,j)$ the transposition of $i$ and $j$. 
It is known that the correspondence 
$s_{1} \mapsto (1, 2),\, 
 s_{2} \mapsto (2, 3),\,\dots,\,
 s_{n} \mapsto (n, n+1)$ 
defines a group isomorphism $W \xrightarrow{\sim} S_{n+1}$, 
where $S_{n+1}$ denotes the symmetric group of degree $n+1$. 
With this isomorphism, we regard $x \in W$ as a permutation on $[n+1]=\{ 1, \ldots, n+1 \}$. 
By this identification, we see that $w \eps_{i} = \eps_{w(i)}$ for $i \in [n+1]$. 
Also, the longest element of $W$, denoted by $\lng$, is regarded as the permutation
\begin{equation*}
\begin{pmatrix}
1 & 2 & \cdots & n & n+1 \\ 
n+1 & n & \cdots & 2 & 1
\end{pmatrix},
\end{equation*}
that is, $\lng$ is considered to be the permutation defined 
by $\lng(k) = n+2-k$ for $k \in [n+1]$. 
%
%=======================%
% START SUBSECTION 0203 %
%=======================%
%
\subsection{An explicit formula expressing the semi-infinite Schubert class associated to $\lng$.}
\label{subsec:expr}

Let $\QGr$ denote the semi-infinite flag manifold associated to $G$, 
which is a pure ind-scheme of infinite type whose set of 
$\BC$-valued points is $G(\BC\pra{z})/(H(\BC) \cdot N(\BC\pra{z}))$ 
(see \cite{KNS, Kat2} for details), 
where $G$ is 
a connected, simply-connected simple algebraic group over $\BC$, $B=HN$ is a Borel subgroup, 
$H$ is a maximal torus, and $N$ is the unipotent radical of $B$; 
note that $\QGr$ is defined as an inductive limit of copies of 
a (reduced) closed subscheme $\QG \subset 
\prod_{i \in I} \BP(L(\vpi_{i}) \otimes_{\BC} \BC\bra{z})$ of infinite type, 
where $L(\vpi_{i})$ is the irreducible highest weight $G$-module of highest weight $\vpi_{i}$. 
One has the semi-infinite Schubert (sub)variety 
$\QG(x) \subset \QGr$ associated to each element $x$ of 
the affine Weyl group $W_{\af} \cong W \ltimes Q^{\vee}$, 
with $W = \langle s_i \mid i \in I \rangle$ 
the (finite) Weyl group and $Q^{\vee} = \sum_{i \in I} \BZ \alpha_i^{\vee}$ 
the coroot lattice of $G$; note that $\QG(x)$ is, by definition, the closure 
of the orbit under the Iwahori subgroup $\bI \subset G(\BC\bra{z})$ 
through the ($H \times \BC^{*}$)-fixed point labeled by $x \in W_{\af}$, 
and that $\QG(x)$ is contained in $\QG(e) = \QG$ 
for all $x \in W_{\af}^{\geq 0} := 
\{ x = w t_{\xi} \in W_{\af} \mid w \in W, \xi \in Q^{\vee,+} \}$, 
where $Q^{\vee, +} := \sum_{i \in I} \BZ_{\geq 0} \alpha_{i}^{\vee} \subset Q^{\vee}$. 
Also, for each weight $\nu = \sum_{i \in I} m_{i} \vpi_{i} \in P$ with $m_{i} \in \BZ$, 
we have a $G(\BC\bra{z}) \rtimes \BC^{*}$-equivariant line bundle 
$\CO_{\QG}(\nu)$ over $\QG$, which is given by the restriction of 
the line bundle $\boxtimes_{i \in I} \CO(m_{i})$ on 
$\prod_{i \in I} \BP(L(\vpi_{i}) \otimes_{\BC} \BC\bra{z})$. 

The $(H \times \BC^{*})$-equivariant $K$-group $K_{H \times \BC^{\ast}}(\QG)$ 
is a module over $\BZ[q,q^{-1}][P]$ (equivariant parameters), 
with the semi-infinite Schubert classes $[\CO_{\QG(x)}]$ associated to 
$x \in W_{\af}^{\geq 0} \simeq W \times Q^{\vee,+}$ as a topological basis 
(in the sense of \cite[Proposition~5.11]{KNS}) over $\BZ[q,q^{-1}][P]$, 
where $P = \sum_{i \in I} \BZ \vpi_{i}$ is the weight lattice of $G$, 
$\BZ[P] = \bigoplus_{\nu \in P} \BZ\be^{\nu} \cong R(H)$, 
and $q \in R(\BC^{*})$ corresponds to loop rotation. 
More precisely, $K_{H \times \BC^{*}}(\QG)$ is defined to be the $\BZ[q, q^{-1}][P]$-submodule 
of the Laurent series (in $q^{-1}$) extension 
$\BZ\pra{q^{-1}}[P] \otimes_{\BZ\bra{q^{-1}}[P]} K_{\bI \rtimes \BC^{*}}^{\prime}(\QG)$ 
of the equivariant, with respect to the Iwahori subgroup $\bI$ and loop rotation, 
$K$-group $K_{\bI \rtimes \BC^{*}}^{\prime}(\QG)$ (see \cite{KNS}) consisting of 
all convergent (in the sense of \cite[Proposition~5.11]{KNS}) 
infinite linear combinations of the classes $[\CO_{\QG(x)}]$, $x \in W_{\af}^{\geq 0}$, 
of the structure sheaves $\CO_{\QG(x)}$ of the semi-infinite Schubert varieties $\QG(x) \subset \QG$ 
with coefficients $a_{x} \in \BZ[q, q^{-1}][P]$; 
briefly speaking, convergence holds 
if the sum $\sum_{x \in W_{\af}^{\geq 0}} \vert a_{x} \vert$ 
of the absolute values $|a_{x}|$ lies in $\BZ_{\geq 0}[P]\pra{ q^{-1} }$. 
Note that for each $x \in W_{\af}^{\geq 0}$ and $\nu \in P$, 
we have the twisted semi-infinite Schubert class 
$[\CO_{\QG(x)}(\nu)] \in K_{H \times \BC^{\ast}}(\QG)$ corresponding to
 the tensor product sheaf $\CO_{\QG(x)} \otimes \CO_{\QG}(\nu)$; 
 in particular, we have $[\CO_{\QG}(\nu)] \in K_{H \times \BC^{\ast}}(\QG)$ for all $\nu \in P$. 

Also, the $H$-equivariant $K$-group $K_{H}(\QG)$ is defined to be the $\BZ[P]$-module 
consisting of all (possibly infinite) linear combinations of 
the semi-infinite Schubert classes $[\CO_{\QG(x)}]$, $x \in W_{\af}^{\geq 0}$, 
with coefficients in $\BZ[P]$, 
which is obtained from $K_{H \times \BC^{*}}(\QG)$ by the specialization $q = 1$; 
it is equipped with tensor product operation 
with an arbitrary line bundle class $[\CO_{\QG}(\nu)]$ for $\nu \in P$. 
Note that $[\CO_{\QG}]$ is the identity for this tensor product operation, 
and hence $[\CO_{\QG}(\nu)] = [\CO_{\QG}] \otimes [\CO_{\QG}(\nu)]$ for all $\nu \in P$. 
In addition, for $\xi \in Q^{\vee,+}$, we can define a $\BZ[P]$-linear operator 
$\st_{\xi}$ on $K_{H}(\QG)$ by $\st_{\xi}[\CO_{\QG(x)}] := 
[\CO_{\QG(x t_{\xi})}]$ for $x \in W_{\af}^{\geq 0}$; 
for $j \in I$, we set $\st_{j} := \st_{\alpha_j^{\vee}}$. 
We deduce that 
\begin{equation}\label{eq:shift}
(\st_{\xi}[\CO_{\QG(x)}]) \otimes [\CO_{\QG}(\nu)] = 
\st_{\xi}([\CO_{\QG(x)}] \otimes [\CO_{\QG}(\nu)])
\end{equation}
for $x \in W_{\af}^{\geq 0}$, $\xi \in Q^{\vee,+}$, and $\nu \in P$; 
in fact, if we define a $\BZ[q, q^{-1}][P]$-linear operator $\st_{\xi}$ on 
$K_{H \times \BC^{\ast}}(\QG)$ in the same way, then we have 
\begin{equation*}
(\st_{\xi}[\CO_{\QG(x)}]) \otimes [\CO_{\QG}(\nu)] = 
q^{-\langle \nu, -w_{\circ} \xi \rangle} \st_{\xi}([\CO_{\QG(x)}] \otimes [\CO_{\QG}(\nu)]) 
\end{equation*}
in $K_{H \times \BC^{\ast}}(\QG)$, 
which follows easily from \cite[Proposition D.1]{KNS}. 
Since the semi-infinite Schubert classes $[\CO_{\QG(x)}]$, $x \in W_{\af}^{\geq 0}$, 
form a topological basis of $K_{H}(\QG)$ over $\BZ[P]$ in the sense above, it follows that
\begin{equation}\label{eq:gshift}
(\st_{\xi} \, \bullet) \otimes [\CO_{\QG}(\nu)] = 
\st_{\xi}(\bullet \otimes [\CO_{\QG}(\nu)])
\end{equation}
for an arbitrary element $\bullet \in K_{H}(\QG)$ and $\xi \in Q^{\vee,+}$, $\nu \in P$.

\begin{rem} \label{rem:st}
Let $\SF_{1},\dots,\SF_{N}$ be polynomials in $\st_{i}$, $i \in I$, 
and let $\nu_{1},\dots,\nu_{N} \in P$. When we write 
$(\SF_{1}\LQG{ \nu_{1} }) \otimes (\SF_{2}\LQG{ \nu_{2} }) 
\otimes \cdots \otimes (\SF_{N}\LQG{ \nu_{N} }) \otimes \bullet$ for $\bullet \in K_{H}(\QG)$, 
it is always understood to be the element 
$(\SF_{1}\SF_{2} \cdots \SF_{N})(\LQG { \nu_{1} + \nu_{2} + \cdots + \nu_{N} } \otimes \bullet)$ 
in $K_{H}(\QG)$. Namely, for $\bullet \in K_{H}(\QG)$, we always understand that 
\begin{equation*}
\begin{split}
& (\SF_{1}\LQG{ \nu_{1} }) \otimes 
  (\SF_{2}\LQG{ \nu_{2} }) \otimes \cdots \otimes
  (\SF_{N}\LQG{ \nu_{N} }) \otimes \bullet \\
 & = (\SF_{1}\SF_{2} \cdots \SF_{N}) 
  (\LQG { \nu_{1} + \nu_{2} + \cdots + \nu_{N} } \otimes \bullet). 
\end{split}
\end{equation*}
\end{rem}

Now, let us assume that $\Fg$ is of type $A_{n}$. 
For $1 \le k \le n$, we set 
\begin{equation*}
\begin{split}
& w_{k}  := 
(s_{1}s_{2} \cdots s_{n})
(s_{1}s_{2} \cdots s_{n-1}) \cdots 
(s_{1}s_{2} \cdots s_{k+1})
(s_{1}s_{2} \cdots s_{k}) \\
& = (1,2,\dots,n+1)(1,2,\dots,n) \cdots (1,2,\dots,k+2)(1,2,\dots,k+1) \\[1.5mm]
& = 
\begin{pmatrix}
1 & 2 & \cdots & k-1 & k & k+1 & \cdots & n & n+1 \\
n-k+2 & n-k+3 & \cdots & n & n+1 & n-k+1  & \cdots & 2 & 1
\end{pmatrix}, 
\end{split}
\end{equation*}
which is an element of the Weyl group $W = S_{n+1}$ of $\Fg$; 
we set $w_{n+1}:=e$ by convention. 
Note that $w_{1} = \lng$, the longest element of $W$. 
Also, for a subset $J$ of $[n+1]=\bigl\{1,2,\dots,n+1\bigr\}$, 
we set 
\begin{equation}
\eps_{J}:=\sum_{j \in J} \eps_{j}. 
\end{equation}
We have $\eps_{J} \in W\vpi_{|J|}$, where 
$\vpi_{0}=\vpi_{n+1}:=0$ by convention. 
In particular, $\eps_{J}$ is a minuscule weight, i.e., 
$\pair{\eps_{J}}{\alpha_{j}^{\vee}} \in \{-1,0,1\}$ for all $1 \le j \le n$. 
We set 
\begin{align*}
J^{-} & := \bigl\{1 \le j \le n \mid j \notin J,\,j+1 \in J \bigr\} \\
& = \bigl\{ 1 \le j \le n \mid \pair{\eps_{J}}{\alpha_{j}^{\vee}} < 0 \bigr\} \\
& = \bigl\{ 1 \le j \le n \mid \pair{\eps_{J}}{\alpha_{j}^{\vee}} = -1 \bigr\}; 
\end{align*}
recall that $\alpha_{j} = \eps_{j}-\eps_{j+1}$. 
For $0 \le p \le k \le n+1$, we set
\begin{equation}\label{eq:fnctn}
\FF^{k}_{p} := 
 \sum_{
   \begin{subarray}{c}
   J \subset [k] \\[1mm]
   |J|=p
   \end{subarray}}
 \left(\prod_{ j \in J^{-} }
 (1-\st_{j})\right)\LQG{ \lng \eps_{J} } \in K_{H}(\QG), 
\end{equation}
where $[k]=\bigl\{1,2,\dots,k\bigr\}$; 
note that $\FF^{k}_{0}=1$. 
%
%%%%%%%%%%%%%%%
%%% prop:wk %%%
%%%%%%%%%%%%%%%
%

The proof of the following proposition will be given in Section~\ref{subsec:prf-wk}. 

\begin{prop} \label{prop:wk}
For $1 \leq k \leq n$, the following equality holds in $K_{H}(\QG)$\,{\rm:} 
\begin{equation} \label{eq:wk}
\left(
 \sum_{p=0}^{k}(-1)^{p} \be^{p \eps_{n+1-k}} \FF^{k}_{p}
\right) \otimes \SQG{w_{k+1}} = \SQG{w_{k}}. 
\end{equation}
\end{prop}

Noting that $w_{1}=\lng$ and $w_{n+1}=e$, we obtain the following. 
%
%%%%%%%%%%%%%%%%%%
%%% cor:wk-lng %%%
%%%%%%%%%%%%%%%%%%
%
\begin{cor} \label{cor:wk-lng}
The following equality holds in $K_{H}(\QG)$\,{\rm:} 
\begin{equation} \label{eq:wk-lng}
\SQG{\lng} = 
\bigotimes_{k=1}^{n}
\left(
 \sum_{p=0}^{k}(-1)^{p} \be^{p \eps_{n+1-k}} \FF^{k}_{p}
\right).
\end{equation}
\end{cor}

%==================%
% START SECTION 03 %
%==================%
%
\section{Proof of the main result for $w = \lng$.} \label{sec03}
%
%=======================%
% START SUBSECTION 0301 %
%=======================%
%
\subsection{Relation between $K_{H}(\QG)$ and $QK_{H}(G/B)$.} \label{sec0301}

Let $G$ be a connected, simply connected simple algebraic group over $\BC$, 
with Borel subgroup $B \subset G$ and maximal torus $H \subset B$. 
Let $QK_{H}(G/B) := K_{H}(G/B) \otimes_{R(H)} R(H)\bra{Q^{\vee,+}}$ 
denote the $H$-equivariant quantum $K$-theory ring of the ordinary flag manifold $G/B$, 
defined by Givental~\cite{Giv} and Lee~\cite{Lee}, where $R(H)\bra{Q^{\vee,+}}$ is 
the ring of formal power series in the Novikov variables $Q_i = Q^{\alpha_i^{\vee}}$, $i \in I$, 
with coefficients in the representation ring $R(H)$ of $H$; 
for $\xi = \sum_{i \in I} k_i \alpha_i^{\vee} \in 
Q^{\vee,+} = \sum_{i \in I} \BZ_{\geq 0} \alpha_i^{\vee}$, we set 
$Q^{\xi} := \prod_{i \in I} Q_i^{k_i} \in R(H)\bra{Q^{\vee,+}}$.
The quantum $K$-theory ring $QK_{H}(G/B)$ is a free module over $R(H)\bra{Q^{\vee,+}}$ 
with the (opposite) Schubert classes $[\CO^{w}]$, $w \in W$, as a basis; 
also, the quantum multiplication $\star$ in $QK_{H}(G/B)$ is a deformation of 
the classical tensor product in $K_{H}(G/B)$, and is defined in terms of the $2$-point and $3$-point 
(genus zero, equivariant) $K$-theoretic Gromov-Witten invariants; 
see \cite{Giv} and \cite{Lee} for details.

In \cite{Kat1,Kat3}, based on \cite{BF,IMT} (see also \cite{ACT}), 
Kato established an $R(H)$-module isomorphism $\Phi$ from $QK_{H}(G/B)$ onto 
the $H$-equivariant $K$-group $K_{H}(\QG)$ of the semi-infinite flag manifold $\QG$, 
in which tensor product operation with an arbitrary line bundle class is 
induced from that in $K_{H\times\BC^*}(\QG)$ by the specialization $q = 1$; 
in our notation, 
the map $\Phi$ sends 
the (opposite) Schubert class $\be^{\mu}[\CO^{w}] Q^{\xi}$ in $QK_{H}(G/B)$ 
to the corresponding semi-infinite Schubert class $\be^{-\mu}[\CO_{\QG(wt_{\xi})}]$ 
in $K_{H}(\QG)$ for $\mu \in P$, $w \in W$, and $\xi \in Q^{\vee,+}$. 
The isomorphism $\Phi$ also respects, in a sense, quantum multiplication $\star$ 
in $QK_{H}(G/B)$ and tensor product in $K_{H}(\QG)$.
More precisely, one has the commutative diagram:
\begin{equation} \label{eq:qdiagram}
\begin{CD}
QK_{H}(G/B) @>{\sim}>> K_{H}(\QG) \\
@V{\bullet \,\, \star [\CO_{G/B}(- \varpi_i)]}VV  
@VV{\bullet \,\, \otimes [\CO_{\QG}(\lng \varpi_i)]}V \\
QK_{H}(G/B) @>>{\sim}> K_{H}(\QG), 
\end{CD}
\end{equation}
where for $\nu \in P$, the line bundle $\CO_{G/B}(-\nu)$ over $G/B$ denotes 
the $G$-equivariant line bundle constructed as 
the quotient space $G \times^{B} \BC_{\nu}$ of 
the product space $G \times \BC_{\nu}$ by the usual (free) left action of 
the Borel subgroup $B$ of $G$ corresponding to the positive roots, 
with $\BC_{\nu}$ the one-dimensional $B$-module of weight $\nu$. 
(Here we warn the reader that the conventions of \cite{Kat1} 
differs from those of \cite{KNS} and this paper,
by the twist coming from the involution $-\lng$.)
Also, we know from \cite{Kat1} that for $w \in W$ and $\xi \in Q^{\vee,+}$, 
$\Phi([\CO^{w}] Q^{\xi}) = \st_{\xi} \, \Phi([\CO^{w}])$ holds, 
and hence that for an arbitrary element $\bullet$ of 
$QK_{H}(G/B)$ and $\xi \in Q^{\vee,+}$, 
\begin{equation} \label{eq:push}
\Phi(\bullet \, Q^{\xi}) = \st_{\xi} \, \Phi(\bullet)
\end{equation}
holds. 

We know from \cite[Corollary~5.14]{BCMP} that 
the quantum multiplicative structure over $R(H)\bra{Q^{\vee,+}}$ 
of $QK_{H}(G/B)$ 
is completely determined by the operators of quantum multiplication 
by $[\CO_{G/B}(-\vpi_{i})]$ for $i \in I$; 
in fact, by using Nakayama's Lemma (i.e., \cite[Lemma A.1]{GMSZ}), 
we can show that $QK_{H}(G/B)$ is generated as an algebra (with quantum multiplication $\star$) 
over $R(H)\bra{Q^{\vee,+}}$ by $[\CO_{G/B}(-\varpi_{i})]$, $i \in I$, 
since $K_{H}(G/B)$ is known to be generated by the same line bundle classes 
as an algebra (with tensor product $\otimes$) over $R(H)$ (\cite{Mi}). 
Therefore, if $G$ is of type $A_{n}$, i.e., $G = SL_{n+1}(\BC)$, 
then we deduce from \cite[Propositions~5.1 and 5.2]{MaNS}, together 
with the comments following them, that $QK_{H}(Fl_{n+1})$ is generated 
as an algebra over $R(H)\bra{Q^{\vee,+}} = R(H)\bra{Q} = R(H)\bra{Q_1, \ldots, Q_{n}}$ 
by the line bundle classes $[\CO_{Fl_{n+1}}(-\eps_{j})]$, $1 \leq j \leq n$; 
indeed, for $1 \leq j \leq n$ and $\bullet \in QK_{H}(Fl_{n+1})$, we have 
\begin{equation*}
\begin{split}
& \bullet \star [\CO_{Fl_{n+1}}(-\varpi_{j})] \\
 & = \bullet \star [\CO_{Fl_{n+1}}(- \epsilon_{1})] \star \frac{1}{1 - Q_{1}}[\CO_{Fl_{n+1}}(- \epsilon_{2})] \star \cdots \star \frac{1}{1 - Q_{j-1}} [\CO_{Fl_{n+1}}(- \epsilon_{j})]. 
\end{split}
\end{equation*}

In the rest of this section, we assume that $G$ is of type $A_{n}$, i.e., $G = SL_{n+1}(\BC)$. 
Now, for $0 \le p \le k \le n+1$, we set
\begin{equation}
\CF^{k}_{p}:=
 \sum_{
   \begin{subarray}{c}
   J \subset [k] \\[1mm]
   |J|=p
   \end{subarray}} \, 
 \prod_{ \begin{subarray}{c} 1 \le j \le k \\[1mm] j,\,j+1 \in J \end{subarray} }
 \frac{1}{1-Q_{j}} 
 \sprod_{j \in J} [\CO_{Fl_{n+1}}(-\eps_{j})] \in QK_{H}(Fl_{n+1}), 
\end{equation}
with $[k]=\bigl\{1,2,\dots,k\bigr\}$, 
where $\prod^{\star}$ denotes the product with respect to the quantum multiplication $\star$; 
note that $\CF^{k}_{0} = 1$ for $0 \le k \le n+1$. 
Also, recall from \eqref{eq:fnctn} that for $0 \leq p \leq k \leq n+1$, 
the element $\FF^{k}_{p} \in K_{H}(\QG)$ is defined by: 
\begin{equation}
\FF^{k}_{p} =
 \sum_{
   \begin{subarray}{c}
   J \subset [k] \\[1mm]
   |J|=p
   \end{subarray}}
 \left( \prod_{ j \notin J,\,j+1 \in J }
 (1-\st_{j}) \right)[\CO_{\QG}(\lng \eps_{J})]. 
\end{equation}
Here we know from \cite[Proposition~5.2]{MaNS}, 
together with the comment following it, that in type $A_{n}$, 
\begin{equation*}
\Phi \left(\bullet \star \frac{1}{1 - Q_{j-1}}[\CO_{Fl_{n+1}}(- \epsilon_{j})]\right) = \Phi(\bullet) \otimes [\CO_{\QG}(\lng \epsilon_{j})]
\end{equation*}
for $1 \leq j \leq n+1$ and an arbitrary element $\bullet \in QK_{H}(Fl_{n+1})$, 
where $G = SL_{n+1}(\BC)$, and $Q_{0} = Q_{n+1} := 0$ by convention. 
By repeated application of this equation, 
together with \eqref{eq:gshift} and \eqref{eq:push}, 
we can show that $\Phi(\CF^{k}_{p}) = \FF^{k}_{p} \in K_{H}(\QG)$ 
for $0 \leq p \leq k \leq n+1$, and also obtain the following proposition 
from Corollary~\ref{cor:wk-lng}.
\begin{prop} \label{prop:lng}
Let $\lng \in W = S_{n+1}$ be the longest element. 
Then, the following equality holds in $QK_{H}(Fl_{n+1})$\,{\rm:}
\begin{equation}
[\CO^{\lng}] = 
\sprod_{k=1}^{n} \Bigg( \sum_{p=0}^{k}(-1)^{p} \be^{-p \eps_{n+1-k}} \CF^{k}_{p}\Bigg) 
\in QK_{H}(Fl_{n+1}).
\end{equation}
\end{prop}
%
%=======================%
% START SUBSECTION 0302 %
%=======================%
%
\subsection{Proof of the main result for $w = \lng$.}
\label{sec0302}

We know from \cite[Section~6]{MaNS} that 
there exists an $R(H)\bra{Q}$-algebra isomorphism 
%
%%%%%%%%%%%%%%%
%%% eq:PsiQ %%%
%%%%%%%%%%%%%%%
%
\begin{equation} \label{eq:PsiQ}
\Psi^{Q}:R(H)\bra{Q}[x_{1},\dots,x_{n},x_{n+1}]/\CI^{Q}
\stackrel{\sim}{\rightarrow} QK_{H}(Fl_{n+1}), 
\end{equation}
where $\CI^{Q}$ denotes the ideal of 
$R(H)\bra{Q}[x_{1},\dots,x_{n},x_{n+1}]$ 
generated by 
%
%%%%%%%%%%%%%
%%% eq:IQ %%%
%%%%%%%%%%%%%
%
\begin{equation} \label{eq:IQ}
\sum_{
   \begin{subarray}{c}
   J \subset [n+1] \\[1mm]
   |J|=l
   \end{subarray}} \, 
\prod_{j \in J, j+1 \notin J} (\ol{1-Q_{j}})
\prod_{j \in J} (1-x_{j}) - 
\sum_{
   \begin{subarray}{c}
   J \subset [n+1] \\[1mm]
   |J|=l
   \end{subarray}}\be^{\eps_{J}}, 
\qquad 1 \le l \le n+1, 
\end{equation}
with $[n+1]=\bigl\{1,2,\dots,n+1\bigr\}$, and 
\begin{equation*}
\ol{1-Q_{j}}=
\begin{cases}
1-Q_{j} & \text{if $1 \le j \le n$}, \\
1 & \text{if $j=n+1$}.
\end{cases}
\end{equation*}
Also, we know that the isomorphism $\Psi^{Q}$ maps the residue class 
$(\ol{1-Q_j})(1-x_{j}) \ \mathrm{mod} \ \CI^{Q}$ to $[\CO_{Fl_{n+1}}(-\eps_{j})]$ 
for $1 \le j \le n+1$; 
note that $-\eps_{n+1} = \eps_1 + \cdots + \eps_n$. 

In the following, by $1 - y_{j} = \be^{-\eps_{j}}$ for $1 \le j \le n$, we identify 
\begin{enu}
\item[(ID1)]
$R(H) \cong \BZ[\be^{\pm \eps_{1}},\dots,\be^{\pm \eps_{n}}]$ 
with $\BZ[(1-y_{1})^{\pm 1},\dots,(1-y_{n})^{\pm 1}]$, and 

\item[(ID2)] $R(H)\bra{Q}[x_{1},\dots,x_{n},x_{n+1}]$ with
\begin{equation} \label{eq:ring}
\BZ\bra{Q}[x_{1},\dots,x_{n},x_{n+1}][(1-y_{1})^{\pm 1},\dots,(1-y_{n})^{\pm 1}].
\end{equation}
\end{enu} 

For $0 \le p \le k \le n$, we define
\begin{equation} \label{eq:F}
F_{p}^{k}(x_{1},\dots,x_{n},x_{n+1}):=
\sum_{
   \begin{subarray}{c}
   J \subset [k] \\[1mm]
   |J|=p
   \end{subarray}} \, 
\prod_{j \in J, j+1 \notin J} (1-Q_{j}) \prod_{j \in J} (1-x_{j}) 
\in \BZ\bra{Q}[x_{1},\dots,x_{n},x_{n+1}]. 
\end{equation}
It easily follows that for $0 \le p \le k \le n$, 
\begin{equation*}
\Psi^{Q}(F_{p}^{k}(x_{1},\dots,x_{n},x_{n+1}) \ \mathrm{mod} \ \CI^{Q}) = 
\CF^{k}_{p} \in QK_{H}(Fl_{n+1}). 
\end{equation*}
Therefore, we deduce from Proposition~\ref{prop:lng} that the residue class 
\begin{equation*}
\prod_{k=1}^{n}
\Bigg( \sum_{p=0}^{k}(-1)^{p}(1-y_{n+1-k})^{p}F_{p}^{k}(x_{1},\dots,x_{n},x_{n+1}) \Bigg) \mod \CI^{Q}
\end{equation*}
is mapped under the $R(H)\bra{Q}$-algebra 
isomorphism $\Psi^{Q}$ to 
\begin{equation*}
\prod_{k=1}^{n}
\Bigg( \sum_{p=0}^{k}(-1)^{p}\be^{-p \eps_{n+1-k}} \CF^{k}_{p} \Biggr) = [\CO^{\lng}] \in QK_{H}(Fl_{n+1}); 
\end{equation*}
recall the identification $1-y_{j} = \be^{-\eps_{j}}$ for $1 \le j \le n$. 

Now, let us set 
\begin{equation} \label{eq:FGlng}
\begin{split}
\FG_{\lng}^{Q}(x,y) & :=
\prod_{k=1}^{n}\Bigg( \sum_{l=0}^{k}(-1)^{l}(1-y_{n+1-k})^{l}F_{l}^{k}(x_{1},\dots,x_{n},x_{n+1}) \Bigg) \\
& \in \BZ\bra{Q}[x_{1},\dots,x_{n},x_{n+1}][(1-y_{1})^{\pm 1},\dots,(1-y_{n})^{\pm 1}]; 
\end{split}
\end{equation}
see Section~\ref{sec0401} below. 
% Then, from the above, we have 
%
% \begin{equation*}
% \Psi^{Q}(\FG_{\lng}^{Q}(x,y) \ \mathrm{mod} \ \CI^{Q}) = [\CO^{\lng}] \in QK_{H}(Fl_{n+1}).
% \end{equation*}
%
Then, from the above, we obtain the following.
%
%%%%%%%%%%%%%%%
%%% thm:lng %%%
%%%%%%%%%%%%%%%
%
\begin{thm} \label{thm:lng}
Let $\lng \in W=S_{n+1}$ be the longest element. 
Then, under the identification {\rm (ID2)}, the following equality holds in $QK_{H}(Fl_{n+1})$\,{\rm:}
\begin{equation}
\Psi^{Q}(\FG_{\lng}^{Q}(x,y) \ \mathrm{mod} \ \CI^{Q}) = [\CO^{\lng}] \in QK_{H}(Fl_{n+1}), 
\end{equation}
where $\Psi^{Q}$ is the $R(H)\bra{Q}$-algebra isomorphism in \eqref{eq:PsiQ}. 
\end{thm}
%
%==================%
% START SECTION 04 %
%==================%
%
\section{Proof of the main result.} \label{sec:main}
%
%=======================%
% START SUBSECTION 0401 %
%=======================%
%
\subsection{Quantum double Grothendieck polynomials.}
\label{sec0401}

First we recall from \cite[Section~8]{LM} 
the Demazure operators $\pi^{(y)}_{i}$ 
acting on the $y$-variables; note that 
we use Demazure operators acting on the $y$-variables, 
while in \cite{LM} they are acting on the $x$-variables. 
Let $i \in [n]=\{1,2,\dots,n\}$. 
We define the Demazure operator $\pi^{(y)}_{i}$ 
on (the fraction field $\Frac(\BZ[y_1^{\pm 1},\dots,y_n^{\pm 1}])$ of) 
$\BZ[y_1^{\pm 1},\dots,y_n^{\pm 1}]$ by:
\begin{equation*}
\pi^{(y)}_{i}=\id + \frac{1-y_{i}}{y_{i}-y_{i+1}}(\id - s_{i});
\end{equation*}
we will further extend it to 
\begin{equation}
\BZ\bra{Q}[x_{1},\dots,x_{n},x_{n+1}] \otimes 
\Frac(\BZ[y_1^{\pm 1},\dots,y_n^{\pm 1}])
\end{equation}
by $\BZ\bra{Q}[x_{1},\dots,x_{n},x_{n+1}]$-linearity. 
Also, for $w \in W=S_{n+1}$ 
with reduced expression $w = s_{i_{1}} \cdots s_{i_{l}}$, we define 
$\pi^{(y)}_{w}:=\pi^{(y)}_{i_1} \cdots \pi^{(y)}_{i_l}$; it is well-known 
that this definition does not depend on the choice of a reduced expression for $w$. 

Now, let us recall from \cite[Section~8]{LM} 
the definition of quantum double Grothendieck polynomials. 
For the longest element $\lng \in W=S_{n+1}$, 
we define $\FG_{\lng}^{Q}(x,y)$ as in \eqref{eq:FGlng}.
Note that the polynomial $\FG_{\lng}^{Q}(x,y)$ is 
the image under the quantization map $\Qq$ (with respect to 
the $x$-variables) of the ordinary double Grothendieck polynomial 
\begin{equation*}
\FG_{\lng}(x,y):=
\prod_{k=1}^{n}\Bigg( \sum_{l=0}^{k} (-1)^{l}(1-y_{n+1-k})^{l}f_{l}^{k}(x_{1},\dots,x_{n},x_{n+1}) \Bigg),
\end{equation*}
where 
\begin{equation*}
f_{l}^{k}(x_{1},\dots,x_{n},x_{n+1}):=
\sum_{
   \begin{subarray}{c}
   J \subset [k] \\[1mm]
   |J|=l
   \end{subarray}} \, \prod_{j \in J}(1-x_{j}), 
\end{equation*}
with $[k] = \{1, 2, \ldots, k\}$. 
For an arbitrary $w \in W=S_{n+1}$, 
the quantum double Grothendieck polynomial $\FG_{w}^{Q}(x,y)$ 
is defined by 
\begin{equation*}
\FG_{w}^{Q}(x,y):=\pi_{w\lng}^{(y)}\FG_{\lng}^{Q}(x,y)
\in \BZ\bra{Q}[x_{1},\dots,x_{n},x_{n+1}][(1-y_{1})^{\pm 1},\dots,(1-y_{n})^{\pm 1}].
\end{equation*}
\begin{rem}
In our definition of quantum double Grothendieck polynomials $\FG_{w}^{Q}(x,y)$, 
the roles of the variables $x$ and $y$ are interchanged from those in \cite[Section~8]{LM}. 
In addition, we use the Demazure operator $\pi^{(y)}_{w\lng}$ instead of $\pi^{(y)}_{w^{-1}\lng}$. 
Hence our polynomial $\FG_{w}^{Q}(x,y)$ coincides with $\FG_{w^{-1}}^{q}(y,x)$ in the notation of 
\cite[Section~8]{LM}.
\end{rem}
\begin{rem}
Our polynomial $\FG_{w}^{Q}(x,y)$ coincides with the one in \cite[Appendix B]{MaNS}; 
this is because the quantization map $\Qq$ (with respect to the $x$-variables) 
commute with the Demazure operators $\pi^{(y)}_{i}$, $1 \leq i \leq n$, and we have 
$\FG_{w}(y,x) = \FG_{w^{-1}}(x,y)$ for ordinary double Grothendieck polynomials.
\end{rem}
%
%=======================%
% START SUBSECTION 0402 %
%=======================%
%
\subsection{Quantum left Demazure operators.}
\label{sec0402}
Let $K_{H}(Fl_{n+1})$ be the (ordinary) $H$-equivariant $K$-theory ring 
of the (full) flag manifold $Fl_{n+1}$ of type $A_{n}$. The ring $K_{H}(Fl_{n+1})$ 
admits a left action of the Weyl group $W=S_{n+1}$, given by the left multiplication 
on the flag manifold $Fl_{n+1}= SL_{n+1}(\BC)/B$, where $B$ is the Borel subgroup of 
$G = SL_{n+1}(\BC)$ consisting of the upper triangular matrices in $G = SL_{n+1}(\BC)$; 
for each $w \in W = S_{n+1}$, let $w^{L}$ denote the corresponding ring 
automorphism of $K_{H}(Fl_{n+1})$. 
By using the ring automorphisms $s_{i}^{L}$, $i \in [n]$, of $K_{H}(Fl_{n+1})$, 
one can define the (dual) left Demazure operators $\delta_{i}^{\vee}$, $i \in [n]$, 
on $K_{H}(Fl_{n+1})$ by 
\begin{equation}
\delta_{i}^{\vee}:=
\frac{1}{ 1-\be^{-(\eps_{i}-\eps_{i+1})} }
(\id - \be^{-(\eps_{i}-\eps_{i+1})}s_{i}^{L}).
\end{equation}

Here we recall the well-known presentation:
\begin{equation*}
K_{H}(Fl_{n+1}) =
(\BZ[\be^{\pm \eps_1},\dots,\be^{\pm \eps_n}] \otimes 
 \BZ[x_{1}^{\pm 1},\dots,x_{n}^{\pm 1}])/\CI,
\end{equation*}
where % $\BZ[X]:=\BZ[x_{1}^{\pm 1},\dots,x_{n}^{\pm 1}]$, and 
$\CI$ is the ideal of $\BZ[\be^{\pm \eps_1},\dots,\be^{\pm \eps_n}] \otimes \BZ[x_{1}^{\pm 1},\dots,x_{n}^{\pm 1}]$
generated by $\iota(f) \otimes 1 - 1 \otimes f$, $f \in \BZ[x_{1}^{\pm 1},\dots,x_{n}^{\pm 1}]^{S_{n}}$, 
with $\iota:\BZ[x_{1}^{\pm 1},\dots,x_{n}^{\pm 1}] \rightarrow \BZ[\be^{\pm \eps_1},\dots,\be^{\pm \eps_n}]$ 
the natural isomorphism given by $\iota(x_{i}^{\pm 1}) = \be^{\pm \eps_i}$, $1 \le i \le n$; 
here we identify the Laurent polynomial ring $\BZ[x_{1}^{\pm 1},\ldots,x_{n}^{\pm 1}]$ 
with the quotient of $\BZ[x_{1}, \ldots, x_{n}, x_{n+1}]$ by the relation $x_{1} \cdots x_{n} x_{n+1} = 1$. 
We know from (the dual version of) \cite[Proposition~9.5]{MNS} that 
under the presentation above of $K_{H}(FL_{n+1})$, 
the action of the (dual) left Demazure operator $\delta_{i}^{\vee}$ on $K_{H}(Fl_{n+1})$ 
coincides with that on the quotient ring 
$(\BZ[\be^{\pm \eps_1},\dots,\be^{\pm \eps_n}] \otimes \BZ[x_{1}^{\pm 1},\dots,x_{n}^{\pm 1}])/\CI$ 
induced by the (dual) Demazure operator on the left component 
$\BZ[\be^{\pm \eps_1},\dots,\be^{\pm \eps_n}]$, given by 
\begin{equation} \label{eq:dem1}
\frac{1}{ 1-\be^{-(\eps_{i}-\eps_{i+1})} }
(\id - \be^{-(\eps_{i}-\eps_{i+1})}s_{i}), 
\end{equation}
where $s_{i}$ denotes the natural action of the simple reflection $s_{i} \in W=S_{n+1}$ 
on $R(H) \cong \BZ[\be^{\pm \eps_1},\dots,\be^{\pm \eps_n}]$. 

\begin{rem} \label{rem:cong}
It is easy to see that 
under the identification (ID1) in Section~\ref{sec0302},
the action of the operator \eqref{eq:dem1} on 
$\BZ[\be^{\pm \eps_1},\dots,\be^{\pm \eps_n}]$ coincides with 
that of the Demazure operator $\pi^{(y)}_{i}$ on 
$\BZ[(1-y_{1})^{\pm 1},\dots,(1-y_{n})^{\pm 1}]$. 
\end{rem}

Let $QK_{H}(Fl_{n+1})=K_{H}(Fl_{n+1}) \otimes_{R(H)} R(H)\bra{Q}$ be 
the $H$-equivariant quantum $K$-theory ring of $Fl_{n+1}$. 
We know from \cite[Section~8]{MNS} that for each $w \in W = S_{n+1}$, 
the ring automorphism $w^{L}$ of $K_{H}(Fl_{n+1})$ extends by $\BZ\bra{Q}$-linearity to 
the ring automorphism of $QK_{H}(Fl_{n+1})$, and hence the (dual) left 
Demazure operators $\delta_{i}^{\vee}$, $1 \leq i \leq n$, on $K_{H}(Fl_{n+1})$ 
extend by $\BZ\bra{Q}$-linearity to $QK_{H}(Fl_{n+1})$. 
Moreover, we know from \cite[Proposition~8.3]{MNS} 
the following equalities in $QK_{H}(Fl_{n+1})$ for $1 \leq i \leq n$: 
\begin{equation} \label{eq:del-com}
\delta_{i}^{\vee}(\bullet \star \CO_{Fl_{n+1}}(\nu)) = 
\delta_{i}^{\vee}(\bullet) \star \CO_{Fl_{n+1}}(\nu)
\end{equation}
for $\nu \in P$ and $\bullet \in QK_{H}(Fl_{n+1})$, and 
\begin{equation} \label{eq:delta}
\delta_{i}^{\vee}([\CO^{w}])=
 \begin{cases}
 [\CO^{s_iw}] & \text{if $s_{i}w < w$}, \\
 [\CO^{w}] & \text{if $s_{i}w > w$},
 \end{cases}
\end{equation}
for $w \in W = S_{n+1}$.  
%
%=======================%
% START SUBSECTION 0403 %
%=======================%
%
\subsection{Proof of the main result.}
\label{sec0403}
Recall the $R(H)\bra{Q}$-algebra isomorphism $\Psi^{Q}$ from \eqref{eq:PsiQ}.
Also, recall the identification (ID1) and (ID2) from Section~\ref{sec0302}.
Let $i \in [n] = \{1, \ldots, n\}$. 
If we extend the Demazure operator $\pi^{(y)}_{i}$ on 
$\BZ[(1-y_{1})^{\pm 1},\dots,(1-y_{n})^{\pm 1}]$ to 
$\BZ\bra{Q}[x_{1},\dots,x_{n},x_{n+1}][(1-y_{1})^{\pm 1},\dots,(1-y_{n})^{\pm 1}]$ by
$\BZ\bra{Q}[x_{1},\dots,x_{n},x_{n+1}]$-linearity, then the (resulting) Demazure operator 
$\pi^{(y)}_{i}$ stabilizes the ideal $\CI^{Q}$ generated by the elements in \eqref{eq:IQ}, 
since the elementary symmetric functions 
$\sum_{J \subset [n+1], \, |J|=l} \, \be^{\eps_{J}}$, $1 \leq l \leq n+1$, 
are invariant under the natural action of $s_i$ 
and the Demazure operator $\pi^{(y)}_{i}$ satisfies the (twisted) Leibniz rule. 
Hence it 
induces a $\BZ\bra{Q}[x_{1},\dots,x_{n},x_{n+1}]$-linear operator on the quotient ring 
$\BZ\bra{Q}[x_{1},\dots,x_{n},x_{n+1}][(1-y_{1})^{\pm 1},\dots,(1-y_{n})^{\pm 1}]/\CI^{Q}$. 
Also, by \eqref{eq:del-com} and the fact that 
$\Psi^{Q}$ maps the residue class of $\ol{(1-Q_{j})}(1-x_{j})$
modulo $\CI^{Q}$ to $[\CO_{Fl_{n+1}}(-\eps_{j})]$, $1 \le j \le n+1$,
it follows that the quantum (dual) left Demazure operator $\delta_{i}^{\vee}$ on 
$QK_{H}(Fl_{n+1})$ is a $\BZ\bra{Q}[x_{1},\dots,x_{n},x_{n+1}]$-linear operator.
Since $QK_{H}(Fl_{n+1})$ and the quotient ring 
$\BZ\bra{Q}[x_{1},\dots,x_{n},x_{n+1}][(1-y_{1})^{\pm 1},\dots,(1-y_{n})^{\pm 1}]/\CI^{Q}$
are both generated by (the image under the quotient map of) 
$\BZ[\be^{\pm \eps_{1}},\dots,\be^{\pm \eps_{n}}] \cong 
 \BZ[(1-y_{1})^{\pm 1},\dots,(1-y_{n})^{\pm 1}]$ over 
$\BZ\bra{Q}[x_{1},\dots,x_{n},x_{n+1}]$, and 
since the (dual) left Demazure operator $\delta_{i}^{\vee}$ and 
the Demazure operator $\pi^{(y)}_{i}$ coincide on 
$\BZ[\be^{\pm \eps_{1}},\dots,\be^{\pm \eps_{n}}] \cong 
 \BZ[(1-y_{1})^{\pm 1},\dots,(1-y_{n})^{\pm 1}]$ by Remark~\ref{rem:cong}, 
together with the comment on $\delta_{i}^{\vee}$ preceding it, 
we deduce that the operators $\delta_{i}^{\vee}$ and $\pi^{(y)}_{i}$ coincide on 
\begin{equation*}
QK_{H}(Fl_{n+1}) \cong 
\BZ\bra{Q}[x_{1},\dots,x_{n},x_{n+1}][(1-y_{1})^{\pm 1},\dots,(1-y_{n})^{\pm 1}]/\CI^{Q}; 
\end{equation*}
recall the identification (ID2). 
Namely, we obtain the following commutative diagram for all $1 \leq i \leq n$: 
\begin{equation} \label{eq:CD}
\begin{CD}
\BZ\bra{Q}[x_{1},\dots,x_{n},x_{n+1}][(1-y_{1})^{\pm 1},\dots,(1-y_{n})^{\pm 1}]/\CI^{Q}
@>{\Psi^{Q}}>{\cong}> QK_{H}(Fl_{n+1}) \\
@V{\pi^{(y)}_{i}}VV @VV{\delta_{i}^{\vee}}V \\
\BZ\bra{Q}[x_{1},\dots,x_{n},x_{n+1}][(1-y_{1})^{\pm 1},\dots,(1-y_{n})^{\pm 1}]/\CI^{Q}
@>{\Psi^{Q}}>{\cong}> QK_{H}(Fl_{n+1}).
\end{CD}
\end{equation}

Now we are ready to state and prove the main result of this paper. 
%
%%%%%%%%%%%%%%%%
%%% thm:main %%%
%%%%%%%%%%%%%%%%
%
\begin{thm} \label{thm:main}
Let $w$ be an arbitrary element of $W = S_{n+1}$. 
Then, under the identification {\rm (ID2)}, 
the following equality holds in $QK_{H}(Fl_{n+1})$\,{\rm :}
\begin{equation*}
\Psi^{Q}(\FG_{w}^{Q}(x,y) \ \mathrm{mod} \ \CI^{Q}) = [\CO^{w}], 
\end{equation*}
where $\Psi^{Q}$ is the $R(H)\bra{Q}$-algebra isomorphism in \eqref{eq:PsiQ}.
\end{thm}

\begin{proof}
The assertion of the theorem is already proved for $w = \lng \in W=S_{n+1}$ by Theorem~\ref{thm:lng}: 
\begin{equation} \label{eq:main1}
\Psi^{Q}(\FG_{\lng}^{Q}(x,y) \ \mathrm{mod} \ \CI^{Q}) = [\CO^{\lng}] \in QK_{H}(Fl_{n+1}). 
\end{equation}
For an arbitrary $w \in W=S_{n+1}$, 
let $w\lng = s_{i_{1}} \cdots s_{i_{l}}$ be a reduced expression; 
notice that $l = \ell(\lng)-\ell(w)$. 
Then, by the definition of quantum double Grothendieck polynomials, we have 
\begin{equation} \label{eq:main2}
\FG_{w}^{Q}(x,y) = \pi_{i_{1}}^{(y)} \cdots \pi_{i_{l}}^{(y)} \FG_{\lng}^{Q}(x,y).
\end{equation}
Here note that $w = s_{i_1} \cdots s_{i_l} \lng$, and that 
$w < s_{i_2} \cdots s_{i_l} \lng < \cdots < s_{i_l}\lng < \lng$ in the Bruhat order $<$ on $W = S_{n+1}$,
since $\ell(w) = \ell(\lng) - l$. Therefore, by the property \eqref{eq:delta} of 
quantum (dual) left Demazure operators $\delta_{i}^{\vee}$, we deduce that 
\begin{equation} \label{eq:main3}
(\delta_{i_1}^{\vee} \cdots \delta_{i_l}^{\vee})[\CO^{\lng}] = [\CO^{w}] \in QK_{H}(Fl_{n+1}).
\end{equation}
From the equalities \eqref{eq:main1}, \eqref{eq:main2}, and \eqref{eq:main3}, 
using the commutative diagram \eqref{eq:CD}, we conclude that 
$\Psi^{Q}(\FG_{w}^{Q}(x,y) \ \mathrm{mod} \ \CI^{Q}) = [\CO^{w}] \in QK_{H}(Fl_{n+1})$. 
This proves the theorem. 
\end{proof}
%
%==================%
% START SECTION 05 %
%==================%
%
\section{Proof of Proposition~\ref{prop:wk}.}
\label{sec:prf-wk}
%
%=======================%
% START SUBSECTION 0501 %
%=======================%
%
\subsection{General Chevalley formula.}

We recall the general Chevalley formula from \cite{LNS} (cf. \cite[Theorem~5.16]{KoLN}), 
which is used in the proof of Proposition~\ref{prop:wk}. 
We use the notation and setting of Section~\ref{subsec:expr}. 
Recall that $\Fg$ is an (arbitrary) finite-dimensional simple Lie algebra over $\BC$. 
Let $\Fh^{\ast}_{\BR} :=\BR \otimes_{\BZ} P$ be a real form of $\Fh^{\ast}$, 
and set 
\begin{equation*}
	H_{\beta,l}:=\bigl\{ \mu \in \Fh^{\ast}_{\BR} \mid \pair{\mu}{\beta^\vee} =l\bigr\} \quad 
	\text{for $\beta \in \Delta$ and $l \in \BZ$}. 
\end{equation*}
We denote by $s_{\beta,l}$ the affine reflection in the affine hyperplane $H_{\beta,l}$.
The affine hyperplanes $H_{\beta,l}$, $\beta \in \Delta$, $l \in \BZ$, divide 
the real vector space $\Fh^{\ast}_{\BR}$ into open regions, called alcoves; 
the fundamental alcove is defined as
\begin{equation*}
	A_{\circ} := \bigl\{ \mu \in \Fh_{\BR}^{\ast} \mid 
	0 < \pair{\mu}{\alpha^\vee} < 1 \quad \text{for all $\alpha \in \Delta^{+}$} \bigr\}.
\end{equation*}
We say that two alcoves are adjacent if they are distinct and have a common wall. Given a pair
of adjacent alcoves $A$ and $B$, we write $A \edge{\beta} B$  for $\beta \in \Delta$ if the
common wall is orthogonal to $\beta$ and $\beta$ points in the direction from $A$ to $B$.

\begin{dfn}[\cite{LP}]
	An alcove path is a sequence of alcoves $(A_0, A_1, \ldots, A_m)$ such that
	$A_{j-1}$ and $A_j$ are adjacent for $j=1,\ldots,m$. 
	We say that $(A_0, A_1, \ldots, A_m)$ is reduced 
	if it has minimal length among all alcove paths from $A_0$ to $A_m$.
\end{dfn}

Let $\lambda \in P$ be a weight, and $A_{\lambda}=A_{\circ}+\lambda$ 
the translation of the fundamental alcove $A_{\circ}$ by the weight $\lambda$.
	
\begin{dfn}[\cite{LP}] \label{dfn:lch}
    Let $\lambda \in P$. 
	The sequence $\Gamma(\lambda)=(\beta_1, \beta_2, \dots, \beta_m)$ of roots is called 
	a reduced $\lambda$-chain (of roots) if 
	\begin{equation}
		A_{\circ}=A_{0} \edge{-\beta_1}  A_1
		\edge{-\beta_2} \cdots 
		\edge{-\beta_m}  A_m=A_{-\lambda}
	\end{equation}
is a reduced alcove path.
\end{dfn}

A reduced alcove path $(A_0=A_{\circ},A_1,\ldots,A_m=A_{-\lambda})$ can be identified 
with the corresponding total order on the hyperplanes, to be called $\lambda$-hyperplanes, 
which separate $A_\circ$ from $A_{-\lambda}$. This total order is given by the sequence 
$H_{\beta_i,-l_i}$ for $i=1,\ldots,m$, where $H_{\beta_i,-l_i}$ contains the common wall of 
$A_{i-1}$ and $A_i$. Note that $\pair{\lambda}{\beta_i^\vee} \ge 0$, and 
that the integers $l_i$, called heights, have the following ranges:
\begin{equation*}
\begin{split}
& 0 \le l_i \le \pair{\lambda}{\beta_i^\vee}-1 \quad \text{if} \quad \beta_i \in \Delta^+, \\
& 1 \le l_i \le \pair{\lambda}{\beta_i^\vee} \quad \text{if} \quad \beta_i \in \Delta^- = - \Delta^+. 
\end{split}
\end{equation*}
Note also that a reduced $\lambda$-chain $(\beta_1, \ldots, \beta_m)$ determines 
the corresponding reduced alcove path, so we can identify them as well. 

Let $\lambda \in P$, and fix an arbitrary $\lambda$-chain 
$\Gamma(\lambda)=(\beta_1,\,\ldots,\,\beta_m)$. 
Also, let $w \in W$. 

\begin{dfn}[\cite{LL}] \label{dfn:admissible}
	A subset 
	$A=\left\{ j_1 < j_2 < \cdots < j_s \right\}$ of $[m]=\{1,\ldots,m\}$ (possibly empty)
 	is a $w$-admissible subset if
	we have the following directed path in the quantum Bruhat graph $\QBG(W)$:
	\begin{equation} \label{eqn:admissible}
	\begin{split}
	& \Pi(w,A): w \edge{|\beta_{j_1}|} w s_{\beta_{j_1}}
	\edge{|\beta_{j_2}|}  ws_{\beta_{j_1}}s_{\beta_{j_2}}
	\edge{|\beta_{j_3}|}  \cdots \\
	& \hspace{30mm} \cdots 
	\edge{|\beta_{j_s}|}  ws_{\beta_{j_1}}s_{\beta_{j_2}} \cdots s_{\beta_{j_s}}=:\ed(A), 
	\end{split}
	\end{equation}
	where for $\beta \in \Delta$, we set 
	\begin{equation*}
	\sgn(\beta):=
	\begin{cases}
	1 & \text{if $\beta \in \Delta^{+}$}, \\
	-1 & \text{if $\beta \in \Delta^{-}$},
	\end{cases}
	\qquad%
	|\beta|:=\sgn(\beta)\beta \in \Delta^{+}.
	\end{equation*}
 	Let $\CA(w,\Gamma(\lambda))$ denote the collection of all $w$-admissible subsets of $[m] = \{1, \ldots, m\}$.
\end{dfn}

Let $A=\{ j_1 < \cdots < j_s\} \in \CA(w,\Gamma(\lambda))$. 
The weight of $A$ is defined by 
	\begin{equation} \label{eq:wta}
	\wt(A):=-w s_{\beta_{j_1},-l_{j_1}} \cdots s_{\beta_{j_s},-l_{j_s}}(-\lambda).
	\end{equation}
Also, we set 
\begin{equation} \label{eq:nA}
n(A):=\# \{ \beta_{j_1},\ldots,\beta_{j_s} \} \cap \Delta^{-}, 
\end{equation}
\begin{equation} \label{eq:A-}
A^{-}:=\left\{j_i \in A \ \biggm| \ 
\begin{array}{l}
\text{%
$ws_{\beta_{j_1}} \cdots s_{\beta_{j_{i-1}}} \edge{|\beta_i|}
 ws_{\beta_{j_1}} \cdots s_{\beta_{j_{i-1}}}s_{\beta_{j_{i}}}$} \\[2mm]
\text{is a quantum edge}
\end{array} \right\}, 
\end{equation}
\begin{equation} \label{def:height}
\dnn(A):=\sum_{j\in A^-}|\beta_j|^\vee\in Q^{\vee,+}.
% \hgt(w,A):=\sum_{j\in A^-}\sgn(\beta_j)\ti{l}_j. 
\end{equation}

Write $\lambda \in P$ as $\lambda=\sum_{i\in I}\lambda_i\vpi_i \in P$, with $\lambda_{i} \in \BZ$ for $i \in I$. 
Let $\ol{\Par(\lambda)}$ denote the set of $I$-tuples of partitions 
$\bchi=(\chi^{(i)})_{i\in I}$ such that $\chi^{(i)}$ is a partition of 
length at most $\max(\lambda_i,0)$.
For $\bchi = (\chi^{(i)})_{i \in I} \in \ol{\Par(\lambda)}$, 
we set $\iota(\bchi) := \sum_{i \in I} \chi^{(i)}_1 \alpha_i^{\vee} \in Q^{\vee,+}$, 
with $\chi^{(i)}_1$ the first part of the partition $\chi^{(i)}$ for each $i \in I$.
%
%%%%%%%%%%%%%%%%%%%
%%% thm:genchev %%%
%%%%%%%%%%%%%%%%%%%
%

By specializing at $q = 1$ in \cite[Theorem~33]{LNS} (cf. \cite[Theorem~5.16]{KoLN}), we obtain the following general Chevalley formula for $K_{H}(\QG)$. 

\begin{thm} \label{thm:genchev} 
Let $\lambda=\sum_{i\in I}\lambda_i\vpi_i \in P$ be an arbitrary weight, 
$\Gamma(\lambda)$ an arbitrary reduced $\lambda$-chain, 
and $x=wt_{\xi}\in W_{\af}^{\ge 0} \simeq W \times Q^{\vee,+}$.
Then, the following equality holds in $K_{H}(\QG)$\,{\rm :}
\begin{align}
& \LQG{-\lng \lambda} \cdot \SQG{x} \nonumber \\[3mm]
& \quad = \sum_{A \in \CA(w,\Gamma(\lambda))}
  \sum_{ \bchi \in \ol{\Par(\lambda)} }
  (-1)^{n(A)} \be^{\wt(A)}
  \SQG{ \ed(A)t_{\xi+\dnn(A)+\iota(\bchi)} }. \label{eq:gchev}
\end{align}
\end{thm}

%%%
%%%
%%%

%=======================%
% START SUBSECTION 0502 %
%=======================%
%
\subsection{Proof of Proposition~\ref{prop:wk}.}
\label{subsec:prf-wk}

%%%
%%%
%%%

This subsection is devoted to the proof of Proposition~\ref{prop:wk}. 
In order to prove equation \eqref{eq:wk}, we will compute
\begin{equation}
\left( 
 \left(\prod_{ j \in J^{-} }
 (1-\st_{j})\right) \LQG{ \lng \eps_{J} }
\right) \otimes \SQG{w_{k+1}} \in K_{H}(\QG)
\end{equation}
for $J = \bigl\{ j_{1},j_{2},\dots,j_{p} \bigr\} \subset [k] =\bigl\{1,2,\ldots,k\bigr\}$, 
with $1 \le j_{1} < j_{2} < \cdots < j_{p} \le k$; 
recall Remark~\ref{rem:st}. 
For this purpose, we make use of Theorem~\ref{thm:genchev}, 
with $\lambda=-\eps_{J}$ and $x = w_{k+1}$. 
%%%%%%
First, notice that for $1 \le j \le n$, 
$\pair{-\eps_{J}}{\alpha_{j}^{\vee}} \in \{-1,0,1\}$, and 
$\pair{-\eps_{J}}{\alpha_{j}^{\vee}} = 1$ if and only if $j \in J^{-}$. 
Hence, in the case that $\lambda=-\eps_{J}$ and $x = w_{k+1}$, 
we can rewrite \eqref{eq:gchev} as 
\begin{align}
& \LQG{\lng \eps_{J}} \cdot \SQG{w_{k+1}} \nonumber \\[3mm]
& \quad = \sum_{A \in \CA(w_{k+1},\Gamma(-\eps_{J}))} \, 
  \sum_{ \bchi \in \ol{\Par(-\eps_{J})} }
  (-1)^{n(A)} \be^{\wt(A)}
  \SQG{ \ed(A)t_{\dnn(A)+\iota(\bchi)} } \nonumber \\[3mm]
& \quad = \sum_{A \in \CA(w_{k+1},\Gamma(-\eps_{J}))} \, 
  \sum_{ \zeta \in \sum_{j \in J^{-}}\BZ_{\ge 0}\alpha_{j}^{\vee} } 
  (-1)^{n(A)} \be^{\wt(A)}
  \SQG{ \ed(A)t_{\dnn(A)+\zeta} } \nonumber \\[3mm]
& \quad = \left(\prod_{j \in J^{-}}
  \frac{1}{1-\st_{j}}\right)
  \sum_{A \in \CA(w_{k+1},\Gamma(-\eps_{J}))}
  (-1)^{n(A)} \be^{\wt(A)}
  \SQG{ \ed(A)t_{\dnn(A)} }. \label{eq:gchev2}
\end{align}

Next, following \cite[Lemma 4.1]{LNOS}, we take a special reduced $(-\eps_{J})$-chain 
(of roots) from $A_{\circ}$ to $A_{\eps_{J}}$ (see Definition~\ref{dfn:lch}) as follows. 
Notice that $\eps_{J} \in W\vpi_{p}$. 
Let $x_{J}$ be the (unique) minimal-length element in 
$\bigl\{ w \in W \mid w \vpi_{p} = \eps_{J} \bigr\}$. 
Then we see that
\begin{equation}
\begin{split}
x_{J} & = s_{i_{a}} \cdots s_{i_{1}} \\
& = 
\underbrace{(s_{j_1-1} \cdots s_{2}s_{1})}_{\text{mapping $\eps_{1}$ to $\eps_{j_{1}}$.}}
\underbrace{(s_{j_2-1} \cdots s_{3}s_{2})}_{\text{mapping $\eps_{2}$ to $\eps_{j_{2}}$.}} 
\cdots \\[2mm]
& \qquad \cdots 
\underbrace{(s_{j_{p-1}-1} \cdots s_{p}s_{p-1})}_{\text{mapping $\eps_{p-1}$ to $\eps_{j_{p-1}}$.}}
\underbrace{(s_{j_{p}-1} \cdots s_{p+1}s_{p})}_{\text{mapping $\eps_{p}$ to $\eps_{j_p}$.}}
\end{split}
\end{equation}
is a reduced expression for $x_{J}$. 
Let $y_{J}$ be the (unique) element such that 
$y_{J}x_{J}$ is the (unique) minimal-length element in 
$\bigl\{ w \in W \mid w \vpi_{p} = \lng \vpi_{p} \bigr\}$. Then, 
\begin{equation}
\begin{split}
y_{J} & = s_{k_{1}} \cdots s_{k_{b}} \\
& = 
\underbrace{(s_{n-p+1} \cdots s_{j_{1}+1}s_{j_{1}})}_{\text{mapping $\eps_{j_1}$ to $\eps_{n-p+2}$.}}
\underbrace{(s_{n-p+2} \cdots s_{j_{2}+1}s_{j_{2}})}_{\text{mapping $\eps_{j_2}$ to $\eps_{n-p+3}$.}} \cdots \\[2mm]
& \qquad \cdots 
\underbrace{(s_{n-1} \cdots s_{j_{p-1}+1}s_{j_{p-1}}) }_{\text{mapping $\eps_{j_{p-1}}$ to $\eps_{n}$.}}
\underbrace{(s_{n} \cdots s_{j_{p}+1}s_{j_{p}})}_{\text{mapping $\eps_{j_p}$ to $\eps_{n+1}$.}}
\end{split}
\end{equation}
is a reduced expression for $y_{J}$. We set 
\begin{align}
& \beta_{c}:=s_{i_{a}} \cdots s_{i_{c+1}}\alpha_{i_{c}} \in \Delta^{+} \quad 
  \text{for $1 \le c \le a$}, \\
& \gamma_{d}:=s_{k_{b}} \cdots s_{k_{d+1}}\alpha_{k_{d}} \in \Delta^{+} \quad 
  \text{for $1 \le d \le b$};
\end{align}
it follows from \cite[Lemma 4.1]{LNOS} that
\begin{equation} \label{eq:GammaJ}
\Gamma_{J}=(\zeta_{1},\dots,\zeta_{a+b}):=
(-\beta_{a},\dots,-\beta_{1},\gamma_{1},\dots,\gamma_{b})
\end{equation}
is a reduced $(-\eps_{J})$-chain from $A_{\circ}$ to $A_{\eps_{J}}$. 
Moreover, for $1 \le c \le a$, the affine hyperplane between the $(c-1)$-th alcove and 
the $c$-th alcove in $\Gamma_{J}$ is $H_{-\beta_{a-c+1},0}$, and 
for $1 \le d \le b$, the affine hyperplane between the $(a+d-1)$-th alcove and 
the $(a+d)$-th alcove in $\Gamma_{J}$ is $H_{\gamma_{d},1}$; note that 
$\eps_{J} \in H_{\gamma_{d},1}$, and hence $s_{\gamma_{d},1}(\eps_{J}) = \eps_{J}$ for all $1 \le d \le b$ 
(see also Remark~\ref{rem:gamma} below). 

%%%%%%%%%%%%%%%%%%%%%%%%%%%%%%%%%%%%%%%%%%%%%%%%%%%%%%%%%
\paragraph{Computation of $(\beta_{a},\dots,\beta_{1})$.}
%%%%%%%%%%%%%%%%%%%%%%%%%%%%%%%%%%%%%%%%%%%%%%%%%%%%%%%%%
%
For $1 \le u \le p$ and $u \le t \le j_{u}-1$, we set
\begin{align*}
\beta_{u,t} & := 
(s_{j_1-1} \cdots s_{2}s_{1}) \cdots
(s_{j_{u-1}-1} \cdots s_{u}s_{u-1})
(s_{j_{u}-1} \cdots s_{t+1})\alpha_{t}; 
\end{align*}
notice that 
the sequence $(\beta_{a},\dots,\beta_{1})$ is identical to
\begin{equation}
\begin{split}
& (\beta_{1,j_{1}-1},\dots,\beta_{1,2},\beta_{1,1},\ 
   \beta_{2,j_{2}-1},\dots,\beta_{2,3},\beta_{2,2},\ \dots\dots, \\
& \ \beta_{p-1,j_{p-1}-1},\dots,\beta_{p-1,p},\beta_{p-1,p-1},\ 
  \beta_{p,j_{p}-1},\dots,\beta_{p,p+1},\beta_{p,p}). 
\end{split}
\end{equation}
Let us compute $\beta_{u,t}$ 
for $1 \le u \le p$ and $u \le t \le j_{u}-1$: 
\begin{align*}
\beta_{u,t} & := 
(s_{j_1-1} \cdots s_{2}s_{1}) \cdots
(s_{j_{u-1}-1} \cdots s_{u}s_{u-1})
(s_{j_{u}-1} \cdots s_{t+1})(\eps_{t}-\eps_{t+1}) \\ 
& = (s_{j_1-1} \cdots s_{2}s_{1}) \cdots
(s_{j_{u-1}-1} \cdots s_{u}s_{u-1}) (\eps_{t}-\eps_{j_{u}}) \\
& =
\begin{cases}
\eps_{t}-\eps_{j_{u}} & \text{if $t > j_{u-1}$}, \\
(s_{j_1-1} \cdots s_{2}s_{1}) \cdots
(s_{j_{u-2}-1} \cdots s_{u-1}s_{u-2})(\eps_{t-1}-\eps_{j_{u}})
& \text{otherwise}. 
\end{cases}
\end{align*}
In the case $t \le j_{u-1}$, we have 
\begin{align*}
& (s_{j_1-1} \cdots s_{2}s_{1}) \cdots
(s_{j_{u-2}-1} \cdots s_{u-1}s_{u-2})(\eps_{t-1}-\eps_{j_{u}}) \\
& = 
\begin{cases}
\eps_{t-1}-\eps_{j_{u}} & \text{if $t-1 > j_{u-2}$}, \\
(s_{j_1-1} \cdots s_{2}s_{1}) \cdots
(s_{j_{u-3}-1} \cdots s_{u-2}s_{u-3})(\eps_{t-2}-\eps_{j_{u}})
& \text{otherwise}. 
\end{cases}
\end{align*}
Repeating this computation, we deduce that
\begin{equation}
\beta_{u,t}=\eps_{t-\tau(u,t)}-\eps_{j_{u}},
\end{equation}
where $\tau(u,t):= |\bigl\{1 \le r \le u-1 \mid t-r+1 \le j_{u-r} \bigr\}|$. 
%
%%%%%%%%%%%%%%%%
%%% rem:beta %%%
%%%%%%%%%%%%%%%%
%
\begin{rem} \label{rem:beta}
Let $1 \le u \le p$, and write 
\begin{equation} \label{eq:t1}
\bigl\{1,2,\dots,j_{u}-1\bigr\} 
\setminus \bigl\{j_{1},\dots,j_{u-1}\bigr\} = 
\bigl\{t_{u,1}< t_{u,2} < \dots < t_{u,m_{u}}\bigr\},
\end{equation}
where $m_{u}=(j_{u}-1)-(u-1) = j_{u}-u$;
we set $t_{u,m_{u}+1}:=j_{u}$ by convention. Then, we have 
\begin{equation*}
\begin{split}
& (\beta_{u,j_{u}-1},\dots,\beta_{u,u+1},\beta_{u,u}) \\
& = \bigl(
\eps_{t_{u,m_u}}-\eps_{j_{u}}, \ \eps_{t_{u,m_u-1}}-\eps_{j_{u}}, \ \dots, \ 
\eps_{t_{u,2}}-\eps_{j_{u}}, \ \eps_{t_{u,1}}-\eps_{j_{u}} \bigr).
\end{split}
\end{equation*}
Hence the sequence $(\beta_{a},\dots,\beta_{1})$ is of the form: 
\begin{align*}
& (j_{1}-1,j_{1}),(j_{1}-2,j_{1}), \dots, (2,j_{1}),(1,j_{1}), \\
& (j_{2}-1,j_{2}),(j_{2}-2,j_{2}), \dots, \text{\sout{$(j_{1},j_{2})$}}, \dots, (2,j_{2}),(1,j_{2}), \\
& (j_{3}-1,j_{3}),(j_{3}-2,j_{3}), \dots, \text{\sout{$(j_{2},j_{3})$}}, \dots, 
  \text{\sout{$(j_{1},j_{3})$}}, \dots, (2,j_{3}), (1,j_{3}), \\
& \dots\dots \\
& (j_{p}-1,j_{p}),(j_{p}-2,j_{p}), \dots, \text{\sout{$(j_{p-1},j_{p})$}}, \dots,
  \text{\sout{$(j_{2},j_{p})$}}, \dots, 
  \text{\sout{$(j_{1},j_{p})$}}, \dots, (2,j_{p}),(1,j_{p}), 
\end{align*}
where $(j_{r},j_{s})$ does not appear 
in the sequence above for any $1 \le r < s \le p$. 
\end{rem}

%%%%%%%%%%%%%%%%%%%%%%%%%%%%%%%%%%%%%%%%%%%%%%%%%%%%%%%%%%%
\paragraph{Computation of $(\gamma_{1},\dots,\gamma_{b})$.}
%%%%%%%%%%%%%%%%%%%%%%%%%%%%%%%%%%%%%%%%%%%%%%%%%%%%%%%%%%%
%
For $1 \le u \le p$ and $j_{u} \le t \le n-p+u$, we set
\begin{align*}
\gamma_{u,t} & := 
(s_{j_{p}} s_{j_{p}+1} \cdots s_{n}) \cdots 
(s_{j_{u+1}} s_{j_{u+1}+1} \cdots s_{n-p+u+1})
(s_{j_{u}} s_{j_{u}+1} \cdots s_{t-1})\alpha_{t}; 
\end{align*}
notice that 
the sequence $(\gamma_{1},\dots,\gamma_{b})$ is identical to
\begin{equation}
\begin{split}
& (\gamma_{1,n-p+1},\dots,\gamma_{1,j_{1}+1},\gamma_{1,j_{1}}, \ 
   \gamma_{2,n-p+2},\dots,\gamma_{2,j_{2}+1},\gamma_{2,j_{2}}, \ \dots\dots, \\
& \ 
   \gamma_{p-1,n-1},\dots,\gamma_{p-1,j_{p-1}+1},\gamma_{p-1,j_{p-1}}, \ 
   \gamma_{p,n},\dots,\gamma_{p,j_{p}+1},\gamma_{p,j_{p}} ). 
\end{split}
\end{equation}
We deduce that for $1 \le u \le p$ and $j_{u} \le t \le n-p+u$, 
\begin{equation}
\gamma_{u,t} := \eps_{j_{u}}-\eps_{t+1+\sigma(u,t)}, 
\end{equation}
where $\sigma(u,t):= |\bigl\{ u+1 \le r \le p \mid t+r-u \ge j_{r} \bigr\}|$. 
%
%%%%%%%%%%%%%%%%%
%%% rem:gamma %%%
%%%%%%%%%%%%%%%%%
%
\begin{rem} \label{rem:gamma}
Let $1 \le u \le p$, and write 
\begin{equation*}
\bigl\{j_{u}+1,j_{u}+2,\dots,n+1 \bigr\} 
\setminus \bigl\{j_{u+1},\dots,j_{p}\bigr\} = 
\bigl\{t_{1}< t_{2} < \dots < t_{m}\bigr\},
\end{equation*}
where $m=(n+1-j_{u})-(p-u) = n + 1 - p - j_{u}+u$. Then, we have 
\begin{equation*}
\begin{split}
& (\gamma_{u,n-p+u},\dots,\gamma_{u,j_{u}+1},\gamma_{u,j_{u}}) \\
& = \bigl( \eps_{j_{u}}-\eps_{t_{m}}, \ \eps_{j_{u}}-\eps_{t_{m-1}}, \ \dots, \ 
\eps_{j_{u}}-\eps_{t_{2}}, \ \eps_{j_{u}}-\eps_{t_{1}} \bigr).
\end{split}
\end{equation*}
Hence the sequence $(\gamma_{1},\cdots,\gamma_{b})$ is of the form: 
\begin{align*}
& (j_{1},n+1),(j_{1},n), \dots, \text{\sout{$(j_{1},j_{p})$}}, \dots,
  \text{\sout{$(j_{1},j_{2})$}}, \dots, (j_{1},j_{1}+2),(j_{1},j_{1}+1), \\ 
& (j_{2},n+1),(j_{2},n), \dots, \text{\sout{$(j_{2},j_{p})$}}, \dots,
  \text{\sout{$(j_{2},j_{3})$}}, \dots, (j_{2},j_{2}+2),(j_{2},j_{2}+1), \\ 
& \dots\dots \\
& (j_{p-1},n+1),(j_{p-1},n), \dots, \text{\sout{$(j_{p-1},j_{p})$}}, \dots, 
  (j_{p-1},j_{p-1}+2),(j_{p-1},j_{p-1}+1), \\ 
& (j_{p},n+1),(j_{p},n), \dots, (j_{p},j_{p}+2),(j_{p},j_{p}+1), 
\end{align*}
where $(j_{r},j_{s})$ does not appear 
in the sequence above for any $1 \le r < s \le p$. 
\end{rem}

Now, let $\bD_{J}$ be the set of directed paths in $\QBG(W)$ of the form: 
\begin{equation} \label{eq:bp}
\begin{split}
& \bp:w_{k+1} = \underbrace{%
  x_{s} \edge{\beta_{a_s}} \cdots \edge{\beta_{a_1}} x_{0}}_{%
  \text{$=:\bp_{\beta}$, the $\beta$-part of $\bp$}} = 
  \underbrace{%
  y_{0} \edge{\gamma_{b_1}} \cdots \edge{\gamma_{b_t}} y_{t}}_{%
  \text{$=:\bp_{\gamma}$, the $\gamma$-part of $\bp$}} = : \ed(\bp),  \\
& \text{\rm with $s \ge 0$ and $a \ge a_{s} > \cdots > a_{1} \ge 1$}, \\
& \text{\rm with $t \ge 0$ and $1 \le b_{1} < b_{2} < \cdots < b_{t} \le b$}; 
\end{split}
\end{equation}
for simplicity of notation, we will henceforth write
\begin{align*}
\bp & = [\beta_{a_s},\,\dots,\,\beta_{a_1} \mid \gamma_{b_1},\,\dots,\gamma_{b_t}], \\
\bp_{\beta} & = [\beta_{a_s},\,\dots,\,\beta_{a_1}], \quad 
\bp_{\gamma} = [\gamma_{b_1},\,\dots,\gamma_{b_t}].
\end{align*}
It is easy to see that 
there exists a bijection between 
the sets $\bD_{J}$ and $\CA(w_{k+1},\Gamma_{J})$. 
Indeed, for a subset $A=\bigl\{p_{1},\dots,p_{u}\}$ of 
$[a+b]=\{1,2,\dots,a+b\}$, $A$ is an element of 
$\CA(w_{k+1},\Gamma_{J})$ if and only if 
\begin{equation} \label{eq:bpa}
\bp(A):w_{k+1}=w_{0} \edge{\zeta_{p_1}} w_{1} \edge{\zeta_{p_2}} \cdots 
\edge{\zeta_{p_u}} w_{u}
\end{equation}
is an element of $\bD_{J}$. 
Notice that $\ed(A)=\ed(\bp(A))$ and $\dnn(A)=\wt(\bp)$. 
For $\bp \in \bD_{J}$ of the form \eqref{eq:bp}, we set
$\ed(\bp_{\beta}):=x_{0}$ and $\ell(\bp_{\gamma}):=t$; 
we see that for $A \in \CA(w_{k+1},\Gamma_{J})$, 
$\wt(A)=-\ed(\bp(A)_{\beta})\eps_{J}$ and $n(A)=\ell(\bp(A)_{\gamma})$. 
Therefore, we deduce from \eqref{eq:gchev2} that in $K_{H}(\QG)$, 
\begin{equation*}
\begin{split}
& \left( 
 \left(\prod_{ j \in J^{-} }
 (1-\st_{j})\right) \LQG{ \lng \eps_{J} }
\right) \otimes \SQG{w_{k+1}} \\[3mm]
& \hspace{5mm} 
= 
 \left(\prod_{ j \in J^{-} }
 (1-\st_{j})\right) \bigl(
 \LQG{ \lng \eps_{J} } \otimes \SQG{w_{k+1}} \bigr) 
 \quad \text{(see Remark~\ref{rem:st})} \\[3mm]
& \hspace{5mm} 
= \left(\prod_{ j \in J^{-} }
  (1-\st_{j})\right)
\left( \left(\prod_{ j \in J^{-} }
  \frac{1}{1-\st_{j}}\right) 
\sum_{\bp \in \bD_{J}} \be^{-\ed(\bp_{\beta})\eps_{J}} (-1)^{\ell(\bp_{\gamma})} 
\SQG{ \ed(\bp)t_{\wt(\bp)} } \right) \\[3mm]
& \hspace{5mm} 
= 
\sum_{\bp \in \bD_{J}} \be^{-\ed(\bp_{\beta})\eps_{J}} (-1)^{\ell(\bp_{\gamma})} 
\SQG{ \ed(\bp)t_{\wt(\bp)} },
\end{split}
\end{equation*}
and hence that 
\begin{align}
& \overbrace{
\left(
 \sum_{p=0}^{k}(-1)^{p} \be^{p \eps_{n+1-k}} \FF^{k}_{p}
\right) \otimes \SQG{ w_{k+1} } }^{ \text{LHS of \eqref{eq:wk}} } \nonumber \\[3mm]
& = 
\sum_{p=0}^{k}(-1)^{p} \be^{p \eps_{n+1-k}}
 \sum_{
   \begin{subarray}{c}
   J \subset [k] \\[1mm]
   |J|=p
   \end{subarray}}
 \left(
 \left(\prod_{ j \in J^{-} }
 (1-\st_{j})\right)\LQG{ \lng \eps_{J} } \right) \otimes \SQG{ w_{k+1} } \nonumber \\[3mm]
& = 
\sum_{J \subset [k]} (-1)^{|J|} \be^{|J| \eps_{n+1-k}}
 \left(
 \left(\prod_{ j \in J^{-} }
 (1-\st_{j})\right)\LQG{ \lng \eps_{J} } \right) \otimes \SQG{ w_{k+1} } \nonumber \\[3mm]
& = \sum_{J \subset [k]} 
\sum_{\bp \in \bD_{J}}
\be^{|J| \eps_{n+1-k}-\ed(\bp_{\beta})\eps_{J}}
(-1)^{|J|+\ell(\bp_{\gamma})} 
\SQG{ \ed(\bp)t_{\wt(\bp)} }. \label{eq:wk2}
\end{align}

We now define
\begin{equation}
\bD:=\bigsqcup_{J \subset [k]} \bD_{J}. 
\end{equation}
For each $J \subset [k]$, we set
\begin{align}
& \bD_{J}^{\RA}:=\bigl\{ \bp \in \bD_{J} \mid \ed(\bp_{\beta})^{-1}(1) \in J,\,
\ed(\bp_{\beta})^{-1}(1) \ne 1 \bigr\}, \\
& \bD_{J}^{\RB}:=\bigl\{ \bp \in \bD_{J} \mid \ed(\bp_{\beta})^{-1}(1) \in J,\,
\ed(\bp_{\beta})^{-1}(1) = 1 \bigr\}, \\
& \bD_{J}^{\RC}:=\bigl\{ \bp \in \bD_{J} \mid \ed(\bp_{\beta})^{-1}(1) \not\in J \bigr\}; 
\end{align}
it is obvious that
$\bD_{J} = \bD_{J}^{\RA} \sqcup \bD_{J}^{\RB} \sqcup \bD_{J}^{\RC}$. Set
\begin{equation}
\bD^{\RX}:=\bigsqcup_{J \subset [k]} \bD_{J}^{\RX} \quad 
\text{for each $\RX \in \{ \RA,\,\RB,\,\RC \}$}; 
\end{equation}
note that
$\bD_{J} = \bD_{J}^{\RA} \sqcup \bD_{J}^{\RB} \sqcup \bD_{J}^{\RC}$.

Here, recall that 
\begin{equation*}
\begin{split}
& w_{k+1}  := 
(s_{1}s_{2} \cdots s_{n})
(s_{1}s_{2} \cdots s_{n-1}) \cdots 
(s_{1}s_{2} \cdots s_{k+1}) \\
& = (1,2,\dots,n+1)(1,2,\dots,n) \cdots (1,2,\dots,k+2) \\[1.5mm]
& = 
\begin{pmatrix}
1 & 2 & \cdots & k & k+1 & k+2 & \cdots & n & n+1 \\
n-k+1 & n-k+2 & \cdots & n & n+1 & n-k  & \cdots & 2 & 1
\end{pmatrix}.
\end{split}
\end{equation*}
Let $J = \bigl\{ j_{1},j_{2},\dots,j_{p} \bigr\} \subset [k]$, 
with $1 \le j_{1} < j_{2} < \cdots < j_{p} \le k$, and define 
$\Gamma_{J}=(-\beta_{a},\ldots,-\beta_{1},\gamma_{1},\dots,\gamma_{b})$ as in \eqref{eq:GammaJ}. 

%%%%%%%%%%%%%%%%%%%%%%%%%%%%%%%%%%%%%%%%%%%%%%%%%%
\paragraph{Computation of the ``$\beta$-part'' of $\bp \in \bD_{J}$.}
%%%%%%%%%%%%%%%%%%%%%%%%%%%%%%%%%%%%%%%%%%%%%%%%%%
%
\begin{enumerate}
\item[(1a)] We see that
\begin{equation} \label{eq:dp1a}
w_{k+1} = x_{0} \edge{(k_{1},j_{1})} x_{1} \edge{(k_{2},j_{1})} \cdots 
\edge{(k_{r},j_{1})} x_{r}
\end{equation}
is a directed path in $\QBG(W)$, where $r \ge 0$ and 
$j_{1}-1 \ge k_{1} > k_{2} > \cdots > k_{r} \ge 1$, 
if and only if $k_{1}=j_{1}-1$, $k_{2}=j_{1}-2$, $\dots$, $k_{r}=j_{1}-r$; 
in this case, all the edges in \eqref{eq:dp1a} are Bruhat edges. 

\item[(1b)] We see that 
\begin{equation} \label{eq:dp1b}
\begin{split}
w_{k+1} = x_{0} & 
\edge{(j_{1}-1,j_{1})} x_{1} \edge{(j_{1}-2,j_{1})} \cdots 
\edge{(j_{1}-r_{1},j_{1})} x_{r_{1}} \\
& 
\edge{(k_{1},j_{2})} x_{r_{1}+1} \edge{(k_{2},j_{2})} \cdots 
\edge{(k_{r},j_{2})} x_{r_{1}+r}
\end{split}
\end{equation}
is a directed path in $\QBG(W)$, where $r_{1},r \ge 0$, 
and $j_{2}-1 \ge k_{1} > k_{2} > \cdots > k_{r} \ge 1$ 
with $k_{r'} \ne j_{1}$ for any $1 \le r ' \le r$, 
if and only if $\{k_{1},k_{2},\dots.k_{r}\}$ is identical to 
the largest $r$ elements in $\{1,2,\dots,j_{2}-1\} \setminus \{j_{1}\}$ (see \eqref{eq:t1}), 
and $k_{r} \ge j_{1}-r_{1}$; 
in this case, all the edges in \eqref{eq:dp1b} are Bruhat edges. 

\item[(1c)] We see that 
\begin{equation} \label{eq:dp1c}
\begin{split}
w_{k+1} = x_{0} & 
\edge{(j_{1}-1,j_{1})} x_{1} \edge{(j_{1}-2,j_{1})} \cdots 
\edge{(j_{1}-r_{1},j_{1})} x_{r_{1}} \\
& 
\edge{(t_{2,m_{2}},j_{2})} x_{r_{1}+1} \edge{(t_{2,m_{2}-1},j_{2})} \cdots 
\edge{(t_{2,m_{2}-r_{2}+1},j_{2})} x_{r_{1}+r_{2}} \\
& 
\edge{(k_{1},j_{3})} x_{r_{1}+r_{2}+1} \edge{(k_{2},j_{3})} \cdots 
\edge{(k_{r},j_{3})} x_{r_{1}+r_{2}+r}
\end{split}
\end{equation}
is a directed path in $\QBG(W)$, where $r_{1},r_{2},r \ge 0$, 
$t_{2,m_{2}-r_{2}+1} \ge j_{1}-r_{1}$, and 
$j_{3}-1 \ge k_{1} > k_{2} > \cdots > k_{r} \ge 1$ 
with $k_{r'} \ne j_{1},j_{2}$ for any $1 \le r ' \le r$, 
if and only if $\{k_{1},k_{2},\dots.k_{r}\}$ is identical to 
the largest $r$ elements in $\{1,2,\dots,j_{3}-1\} \setminus \{j_{1},j_{2}\}$ (see \eqref{eq:t1}), 
and $k_{r} \ge t_{2,m_{2}-r_{2}+1}$; 
in this case, all the edges in \eqref{eq:dp1c} are Bruhat edges. 
\end{enumerate}

Repeating the argument above, we obtain the following. 
%
%%%%%%%%%%%%%%%%%
%%% lem:bpart %%%
%%%%%%%%%%%%%%%%%
%
\begin{lem} \label{lem:bpart}
{\rm (1)} We have a directed path in $\QBG(W)$ of the form
\begin{equation} \label{eq:dp1p}
\begin{split}
w_{k+1} & 
\Be{(t_{1,m_{1}},j_{1})} \bullet \Be{(t_{1,m_{1}-1},j_{1})} \cdots 
\Be{(t_{1,m_{1}-r_{1}+1},j_{1})} \bullet \\
& 
\Be{(t_{2,m_{2}},j_{2})} \bullet \Be{(t_{2,m_{2}-1},j_{2})} \cdots 
\Be{(t_{2,m_{2}-r_{2}+1},j_{2})} \bullet \\
& \cdots \cdots \\
&
\Be{(t_{p,m_{p}},j_{p})} \bullet \Be{(t_{p,m_{p}-1},j_{p})} \cdots 
\Be{(t_{p,m_{p}-r_{p}+1},j_{p})} \bullet
\end{split}
\end{equation}
for each $r_{1},r_{2},\dots,r_{p} \ge 0$ and 
$t_{p,m_{p}-r_{p}+1} \ge \cdots \ge t_{2,m_{2}-r_{2}+1} \ge t_{1,m_{1}-r_{1}+1}$. 
% where $t_{u,m_{u}+1}:=j_{u}$ for $1 \le u \le p$ by convention. 

{\rm (2)} The $\beta$-part of 
an element $\bp \in \bD_{J}$ is of the form \eqref{eq:dp1p}. 
\end{lem}

Let $\bp \in \bD_{J}$. We can deduce the following:
\begin{itemize}
\item $\bp \in \bD_{J}^{\RA}$ if and only if 
$J \ne \emptyset$, with $j_{1} = \min J \ge 2$ (and hence $1 \notin J$), 
and the $\beta$-part $\bp_{\beta}$ is of the form: 
\begin{align}
\bp_{\beta} & = [ (j_{1}-1,j_{1}), (j_{1}-2,j_{1}),\dots,(2,j_{1}),(1,j_{1}), \nonumber \\
& \underbrace{%
  (t_{2,m_{2}},j_{2}),\dots,(t_{2,m_{2}-r_{2}+1},j_{2}), \dots\dots, 
  (t_{p,m_{p}},j_{p}),\dots,(t_{p,m_{p}-r_{p}+1},j_{p})}_{=:\bs}] \label{eq:pbA}
\end{align}
for some $r_{2},\dots,r_{p} \ge 0$ and 
$t_{p,m_{p}-r_{p}+1} \ge \cdots \ge t_{2,m_{2}-r_{2}+1} \ge 1$; 
note that $\ed(\bp_{\beta})^{-1}(1) = j_{1} \in J$ since 
$(j_{1},j_{r})$ does not appear in $\bs$ for any $2 \le r \le p$. 

\item $\bp \in \bD_{J}^{\RB}$ if and only if $1 \in J$; note that 
$j_{1}=\min J$ is equal to $1$, and 
the $\beta$-part $\bp_{\beta}$ is of the form: 
\begin{align}
\bp_{\beta} & = [ 
  (t_{2,m_{2}},j_{2}),\dots,(t_{2,m_{2}-r_{2}+1},j_{2}), \nonumber \\
  & \qquad \dots\dots, 
  (t_{p,m_{p}},j_{p}),\dots,(t_{p,m_{p}-r_{p}+1},j_{p})] \label{eq:pbB}
\end{align}
for some $r_{2},\dots,r_{p} \ge 0$ and 
$t_{p,m_{p}-r_{p}+1} \ge \cdots \ge t_{2,m_{2}-r_{2}+1} \ge 2$. 

\item $\bp \in \bD_{J}^{\RC}$ if and only if $1 \not\in J$ (and hence 
$j_{1}=\min J \ge 2$ if $J \ne \emptyset$), and 
the $\beta$-part $\bp_{\beta}$ is of the form: 
\begin{align}
\bp_{\beta} & = [ (j_{1}-1,j_{1}), (j_{1}-2,j_{1}),\dots,(j_{1}-r_{1}+1,j_{1}),(j_{1}-r_{1},j_{1}), \nonumber \\
& (t_{2,m_{2}},j_{2}),\dots,(t_{2,m_{2}-r_{2}+1},j_{2}), \dots\dots, 
  (t_{p,m_{p}},j_{p}),\dots,(t_{p,m_{p}-r_{p}+1},j_{p})] \label{eq:pbC}
\end{align}
for some $r_{1},r_{2},\dots,r_{p} \ge 0$ and 
$t_{p,m_{p}-r_{p}+1} \ge \cdots \ge t_{2,m_{2}-r_{2}+1} \ge j_{1} - r_{1} \ge 2$; 
note that $\ed(\bp_{\beta})^{-1}(1) = 1 \notin J$. 
\end{itemize}

%%%%%%%%%%%%%%%%%%%%%%%%%%%%%%%%%%%%%%%%%%%%%%%%%%%%%%%%%
\paragraph{Computation of the ``$\gamma$-part'' of $\bp \in \bD_{J}^{\RA}$.}
%%%%%%%%%%%%%%%%%%%%%%%%%%%%%%%%%%%%%%%%%%%%%%%%%%%%%%%%%
%
Let $J = \bigl\{ j_{1},j_{2},\dots,j_{p} \bigr\} \subset [k]$, 
with $1 \le j_{1} < j_{2} < \cdots < j_{p} \le k$, as above. 
Let $\bp \in \bD_{J}^{\RA} \sqcup \bD_{J}^{\RB}$, and assume that 
the $\gamma$-part $\bp_{\gamma}$ of $\bp$ is of the form:
\begin{equation*}
\begin{split}
\ed(\bp_{\beta}) & 
\edge{(j_{1},k_{r})} \bullet \edge{(j_{1},k_{r-1})} \cdots 
\edge{(j_{1},k_{2})} \bullet \edge{(j_{1},k_{1})} \bullet \\
& \cdots \cdots, 
\end{split}
\end{equation*}
where $r \ge 0$, and 
$n+1 \ge k_{r} > k_{r-1} > \cdots > k_{1} \ge j_{1}+1$ 
with $k_{u} \notin \{j_{2},\dots,j_{p}\}$ for any $1 \le u \le r$. 
%
%%%%%%%%%%%%%%%%%
%%% lem:gamma %%%
%%%%%%%%%%%%%%%%%
%
\begin{lem} \label{lem:gamma}
Keep the notation and setting above. 
Either of the following holds\,{\rm:} {\rm (i)} $r = 0$, or 
{\rm (ii)} $r = 1$ and $k_{1}=j_{1}+1$\,{\rm;} 
note that {\rm (ii)} does not occur if $j_{1}+1 \in J$. 
In case {\rm (ii)}, the edge $\ed(\bp_{\beta}) 
\edge{(j_{1},j_{1}+1)} \bullet$ is a Bruhat edge. 
\end{lem}

\begin{proof}
Assume that $r > 0$. 
Note that $\ed(\bp_{\beta})(j_{1}) = w_{k+1}(1) = n-k+1$. 
If $k_{r} \ge k+2$, then 
\begin{equation*}
\ed(\bp_{\beta})(k_{r}) = 
w_{k+1}(k_{r}) = n-k_{r}+2 < n-k+1 = \ed(\bp_{\beta})(j_{1}).
\end{equation*}
However, $j_{1} < k+1 < k_{r}$, and 
$\ed(\bp_{\beta})(k+1) = n+1$ is greater than 
both $\ed(\bp_{\beta})(k_{r})$ and $\ed(\bp_{\beta})(j_{1})$, 
which is a contradiction. Hence we get $k_{r} \le k+1$. 

Suppose, for a contradiction, that $j_{1}+1=j_{2} \in J$; 
note that $j_{2} < k_{r} \le k+1$ and $k_{r} \notin \{j_{3},\dots,j_{p}\}$. 
In this case, it follows that 
\begin{equation*}
\ed(\bp_{\beta})(j_{2}) = w_{k+1}(j_{2}) = n-k+j_{2} > n-k+1 = \ed(\bp_{\beta})(j_{1}). 
\end{equation*}
Since $k_{r} > j_{2}$, we see that 
$\ed(\bp_{\beta})(k_{r}) = w_{k+1}(t)=n-k+t$ for some $j_{2} < t \le k+1$. 
Hence we obtain
\begin{equation*}
\ed(\bp_{\beta})(j_{1}) < \ed(\bp_{\beta})(j_{2}) < \ed(\bp_{\beta})(k_{r}),
\end{equation*}
which is a contradiction. 

Suppose, for a contradiction, that $j_{1}+1 \notin J$ and $j_{1}+1 < k_{r} \le k+1$; 
recall that $k_{r} \notin \{j_{2},\dots,j_{p}\}$. Also, in this case, 
we see by \eqref{eq:pbA} that 
\begin{equation*}
\ed(\bp_{\beta})(j_{1})=1 < \ed(\bp_{\beta})(j_{1}+1) < \ed(\bp_{\beta})(k_{r}), 
\end{equation*}
which is a contradiction. Hence we get $k_{r}=j_{1}+1$; 
since $\ed(\bp_{\beta})(j_{1})=1 < \ed(\bp_{\beta})(j_{1}+1)$, 
the edge $\ed(\bp_{\beta}) 
\edge{(j_{1},j_{1}+1)} \bullet$ is a Bruhat edge. 
This proves the lemma. 
\end{proof}

Let $\bD_{J}^{\RA_1}$ (resp., $\bD_{J}^{\RA_2}$) denote 
the subset of $\bD_{J}^{\RA}$ consisting of all $\bp \in \bD_{J}^{\RA}$ 
whose initial label of the $\gamma$-part $\bp_{\gamma}$ is (resp., is not) 
$(j_{1},j_{1}+1)$. 
% Note that $\bD_{J}^{\RA_1} = \emptyset$ if and only if $j_{1}+1 \in J$. 
Similarly, let $\bD_{J}^{\RB_1}$ (resp., $\bD_{J}^{\RB_2}$) denote 
the subset of $\bD_{J}^{\RB}$ consisting of all $\bp \in \bD_{J}^{\RB}$ 
whose initial label of the $\gamma$-part $\bp_{\gamma}$ is (resp., is not) $(1,2)$. 
% Note that $\bD_{J}^{\RB_1} = \emptyset$ if and only if $2 \in J$. 

We can easily show the following.
%
%%%%%%%%%%%%%
%%% lem:q %%%
%%%%%%%%%%%%%
%
\begin{lem} \label{lem:q}
We have 
\begin{equation}
\bq:=[(k-1,k),(k-2,k),\dots,(1,k) \mid (k,k+1)] \in 
\bD_{\{k\}}^{\RA_1}; 
\end{equation}
note that $\{k\} \subset [k]$ is a unique subset $J \subset [k]$ such that 
$\bq \in \bD_{J}^{\RA_1}$. Also, we have 
\begin{equation}
\ed(\bq_{\beta})\eps_{J} = \eps_{n+1-k}, \quad 
\ell(\bp_{\gamma})=1, \quad 
\ed(\bp)t_{\wt(\bp)} = w_{k}.
\end{equation}
\end{lem}

In the following, we will define a ``sijection'' (i.e., bijection between signed sets) $\Phi:\bD \rightarrow \bD$. 
We set $\Phi(\bq):=\bq$ (see Lemma~\ref{lem:q}). Let $\bp \in \bD_{J}^{\RA_2}$; 
note that $j_{1}=\min J \ge 2$ and $j_{1}-1 \notin J$. We set 
$\Phi(J):=(J \setminus \{j_{1}\}) \sqcup \{j_{1}-1\}$, and then define 
$\Phi(\bp) \in \bD_{\Phi(J)}^{\RA_1} \sqcup \bD_{\Phi(J)}^{\RB_1}$ as follows. 
Recall that the $\beta$-part $\bp_{\beta}$ of $\bp$ is of the form \eqref{eq:pbA}, 
and the $\gamma$-part $\bp_{\gamma}$ of $\bp$ starts at the edge whose label is 
$(j_{u},t)$ for some $2 \le u \le p$ and $j_{u} \le t \le n+1$, with 
$t \notin \{j_{u+1},\dots,j_{p}\}$ (if $\bp_{\gamma}$ is nontrivial). 
We define $\Phi(\bp)$ by: 
% removing $(j_{1}-1,j_{1})$ from $\bp_{\beta}$, 
% replacing all $j_{1}$ appearing in $\bp_{\beta}$ with $j_{1}-1$, 
% replacing all $j_{1}$ appearing in $\bp_{\beta}$ with $j_{1}-1$, 
% and then adding $(j_{1}-1,j_{1})$ at the beginning of $\bp_{\gamma}$: 
%
\begin{equation} \label{eq:PhiA1}
\Phi(\bp) = [ (j_{1}-2,j_{1}-1),\dots,(2,j_{1}-1),(1,j_{1}-1), \bs' \mid 
(j_{1}-1,j_{1}), \bp_{\gamma} ], 
\end{equation}
where $\bs'$ is obtained from $\bs$ by replacing $(j_{1}-1,j_{u})$ 
appearing in $\bp_{\beta}$ with $(j_{1},j_{u})$ for each $2 \le u \le p$.
By Lemma~\ref{lem:bpart}, it is easily verified that 
$\Phi(\bp) \in \bD_{\Phi(J)}^{\RA_1} \sqcup \bD_{\Phi(J)}^{\RB_1}$ 
(note that $\Phi(\bp) \in \bD_{\Phi(J)}^{\RB_1}$ if and only if $j_{1}=2$). 
Also, it follows that 
\begin{equation} \label{eq:PhiA1a}
\begin{split}
& |\Phi(J)|=|J|, \\
& |\Phi(J)|\eps_{n+1-k}-\ed(\Phi(\bp)_{\beta})\eps_{\Phi(J)} =|J|\eps_{n+1-k}-\ed(\bp_{\beta})\eps_{J}, \\
& \ell(\Phi(\bp)_{\gamma})=\ell(\bp_{\gamma})+1, \quad 
  \ed(\Phi(\bp))t_{\wt(\Phi(\bp))} = \ed(\bp)t_{\wt(\bp)}.
\end{split}
\end{equation}

Let $\bp \in \bD_{J}^{\RA_1} \sqcup \bD_{J}^{\RB_1}$, with $J \ne \{k\}$; 
recall that $\bD_{ \{ k \} }^{\RA_1} \sqcup \bD_{ \{ k \} }^{\RB_1} = \{\bq\}$, and note that $j_{1}+1 \notin [k] \setminus J$. 
We set $\Phi(J):= (J \setminus \{j_{1}\}) \sqcup \{j_{1}+1\}$, and then 
define $\Phi(\bp) \in \bD_{\Phi(J)}^{\RA_1}$ as follows. 
Recall that the $\beta$-part $\bp_{\beta}$ of $\bp$ 
is of the form \eqref{eq:pbA} or \eqref{eq:pbB}, 
and the $\gamma$-part $\bp_{\gamma}$ of $\bp$ 
starts at the edge whose label is $(j_{1},j_{1}+1)$. 
We define $\Phi(\bp)$ by: 
\begin{equation} \label{eq:PhiB1}
\begin{split}
\Phi(\bp) & = [ (j_{1},j_{1}+1), (j_{1}-1,j_{1}+1),\dots,(2,j_{1}+1),(1,j_{1}+1), \bs'' \\
& \hspace{70mm} \mid \bp_{\gamma} \setminus (j_{1},j_{1}+1) ], 
\end{split}
\end{equation}
where $\bs''$ is obtained from $\bs$ by replacing 
$(j_{1}+1,j_{u})$ appearing in $\bp_{\beta}$ with $(j_{1},j_{u})$ 
for each $2 \le u \le p$, and $\bp_{\gamma} \setminus (j_{1},j_{1}+1)$ 
is the sequence obtained form $\bp_{\gamma}$ by removing the initial label 
$(j_{1},j_{1}+1)$. By Lemma~\ref{lem:bpart}, it is easily verified that 
$\Phi(\bp) \in \bD_{\Phi(J)}^{\RA_2}$. Also, it follows that 
\begin{equation} \label{eq:PhiB1a}
\begin{split}
& |\Phi(J)|=|J|, \\
& |\Phi(J)|\eps_{n+1-k}-\ed(\Phi(\bp)_{\beta})\eps_{\Phi(J)} = 
  |J|\eps_{n+1-k}-\ed(\bp_{\beta})\eps_{J}, \\
& \ell(\Phi(\bp)_{\gamma})=\ell(\bp_{\gamma})-1, \quad 
  \ed(\Phi(\bp))t_{\wt(\Phi(\bp))} = \ed(\bp)t_{\wt(\bp)}.
\end{split}
\end{equation}

If $\bp \in \bD_{J}^{\RB_2}$, then $\bp \in \bD_{\Phi(J)}^{\RC}$ 
with $\Phi(J):=J \setminus \{1\}$. We define $\Phi(\bp):=\bp \in 
\bD_{\Phi(J)}^{\RC}$ for $\bp \in \bD_{J}^{\RB_2}$. It follows that 
\begin{equation} \label{eq:PhiC1}
\begin{split}
& |\Phi(J)|=|J|-1, \\
& |\Phi(J)|\eps_{n+1-k}-\ed(\Phi(\bp)_{\beta})\eps_{\Phi(J)} = 
  |J|\eps_{n+1-k}-\ed(\bp_{\beta})\eps_{J}, \\
& \ell(\Phi(\bp)_{\gamma})=\ell(\bp_{\gamma}), \qquad 
  \ed(\Phi(\bp))t_{\wt(\Phi(\bp))} = \ed(\bp)t_{\wt(\bp)}.
\end{split}
\end{equation}
Similarly, if $\bp \in \bD_{J}^{\RC}$, then $\bp \in \bD_{\Phi(J)}^{\RB_2}$ 
with $\Phi(J):=J \sqcup \{1\}$. We define $\Phi(\bp):=\bp \in 
\bD_{\Phi(J)}^{\RB_2}$ for $\bp \in \bD_{J}^{\RC}$. It follows that 
\begin{equation} \label{eq:PhiC2}
\begin{split}
& |\Phi(J)|=|J|+1, \\
& |\Phi(J)|\eps_{n+1-k}-\ed(\Phi(\bp)_{\beta})\eps_{\Phi(J)} =
  |J|\eps_{n+1-k}-\ed(\bp_{\beta})\eps_{J}, \\
& \ell(\Phi(\bp)_{\gamma})=\ell(\bp_{\gamma}), \qquad 
  \ed(\Phi(\bp))t_{\wt(\Phi(\bp))} = \ed(\bp)t_{\wt(\bp)}.
\end{split}
\end{equation}

Thus, we have obtained a bijection $\Phi:\bD \rightarrow \bD$ satisfying 
(i) $\Phi(\bq)=\bq$, (ii) 
$\Phi(\bp) \ne \bp$, $\Phi(\Phi(\bp))=\bp$, and 
\begin{equation*}
\begin{split}
& \be^{|\Phi(J)| \eps_{n+1-k}-\ed(\Phi(\bp)_{\beta})\eps_{\Phi(J)}}
  (-1)^{|\Phi(J)|+\ell(\Phi(\bp)_{\gamma})} 
  [\CO_{\QG(\ed(\Phi(\bp))t_{\wt(\Phi(\bp))})}] \\
& = - \be^{|J| \eps_{n+1-k}-\ed(\bp_{\beta})\eps_{J}}
 (-1)^{|J|+\ell(\bp_{\gamma})} 
 [\CO_{\QG(\ed(\bp)t_{\wt(\bp)})}]
\end{split}
\end{equation*}
for all $\bp \in \bD_{J} \setminus \{\bq\}$ with $J \subset [k]$. 
Therefore, the right-hand side of equation \eqref{eq:wk2} becomes 
\begin{align*}
& \be^{|\{k\}| \eps_{n+1-k}-\ed(\bq_{\beta})\eps_{\{k\}}}
 (-1)^{|\{k\}|+\ell(\bq_{\gamma})} 
 [\CO_{\QG(\ed(\bq)t_{\wt(\bq)})}] \\
& = [\CO_{\QG(w_{k})}] \qquad \text{by Lemma~\ref{lem:q}}, 
\end{align*}
as desired. 
This completes the proof of Proposition~\ref{prop:wk}. 

%=====================%
%     BIBLIOGRAPHY    %
%=====================%

{\small

}

\end{document}